\documentclass[11pt]{amsart}

\usepackage[margin=1.05in]{geometry}
\usepackage{amsmath,amssymb,amsthm,mathtools}
\usepackage{microtype}
\usepackage[hidelinks]{hyperref}

\numberwithin{equation}{section}

\newtheorem{theorem}{Theorem}[section]
\newtheorem{proposition}[theorem]{Proposition}
\newtheorem{lemma}[theorem]{Lemma}
\newtheorem{remark}[theorem]{Remark}
\newtheorem{corollary}[theorem]{Corollary}

\usepackage{xcolor}

\newcommand{\R}{\mathbb R}
\newcommand{\Z}{\mathbb Z}
\newcommand{\Q}{\mathbb Q}
\newcommand{\N}{\mathbb N}
\newcommand{\T}{\mathbb T}
\newcommand{\SL}{\mathrm{SL}}
\newcommand{\SO}{\mathrm{SO}}
\newcommand{\sys}{\operatorname{sys}}
\newcommand{\tr}{\operatorname{tr}}

\newcommand{\cl}{\overline}

\title{Existence of a Hall ray in higher dimensional Lagrange spectra}
\author{Yitwah Cheung}
\address[Cheung]{Yau Mathematical Sciences Center, Tsinghua University, Beijing, China}
\email{yitwah@tsinghua.edu.cn}

\author{Zhijing Wendy Wang}
\address[Wang]{Department of Mathematics, University of Chicago, Chicago, Illinois 60637, United States of America}
\email{zhijingw@uchicago.edu}

\author{Ruichong Zhang}
\address[Zhang]{Qiuzhen College, Tsinghua University, Beijing, China}
\email{zhangrc24@mails.tsinghua.edu.cn}

\date{September 18, 2026}

\begin{document}
\maketitle

\begin{abstract}
Given any norm $\|\cdot\|_{\R^2}$ on $\R^2$, we prove that the Lagrange spectrum for $2$-vectors contains a Hall ray. Equivalently, every sufficiently small $\rho\ge0$ occurs as
\[
\liminf_{q\to\infty}q^{1/2}\min_{p\in\Z^2}\|qx-p\|_{\R^2}
\]
for some $x\in\R^2\setminus\Q^2$. Combined with a recent result of Kleinbock \cite{Kleinbock}, this implies that the $2$-dimensional Lagrange spectrum is either a ray, mirroring the behavior for Dirichlet spectra, or a ray preceded by a dense part.  
\end{abstract}

\setcounter{tocdepth}{1}

\section{Introduction}
For $\alpha\in \R$, write
\[
 \|\alpha\|=\min_{p\in\Z}|\alpha-p|.
\]
The classical Lagrange spectrum is
\[
 \mathcal L_1
 =
 \{\big(\liminf_{q\to\infty}q\,\|q\alpha\|\big)^{-1}:
   \alpha\in\R\}.
\]
This 1-dimensional Lagrange spectrum $\mathcal L_1$ has several sharply different regimes. 
At the lower end, Markov~\cite{Markov} proved that
\[
 \mathcal L_1\cap(-\infty,3)
 =
 \left\{
   \sqrt{9-\frac{4}{m^2}}:
   m\text{ is a Markov number}
 \right\},
\]
which is a discrete sequence beginning with
$
 \sqrt5,
 2\sqrt2,
 \frac{\sqrt{221}}5,\ldots $
and accumulating at $3$ \cite[pp.~145--147]{Moreira}. 
At the opposite end, Hall showed that the spectrum eventually contains no gaps.

\begin{theorem}[Hall~\cite{Hall}]
There is a constant $L_0$ such that every $L\geq L_0$ belongs to the classical one-dimensional Lagrange spectrum.
\end{theorem}

The half-line $[L_0,\infty)$ is called a \emph{Hall ray}. Freiman~\cite{Freiman} later determined the left endpoint
\[
 c_F\approx4.527829566
\]
of the largest such half-line \cite[pp.~146--148]{Moreira}. 
Between these two regions,
the spectrum has a rich fractal structure; see~\cite{Moreira,ErazoEtAl}.

In this paper, we study the corresponding spectrum for simultaneous approximation. Fix a norm $\|\cdot\|_{\R^d}$ on $\R^d$ and
write
\[
 \|x\|_{\T^d}=\min_{p\in\Z^d}\|x-p\|_{\R^d}.
\]
Define the $d$-\textbf{dimensional Lagrange spectrum} to be
\[
 \mathcal L_d := \{\big(\liminf_{q\to \infty} q^{1/d}\|q\theta\|_{\T^d}\big)^{-1} :\theta\in\R^d\}.
\]

In this paper, we show that in dimension $d=2$, the Lagrange spectrum $\mathcal{L}_2$ contains a ray. 

\begin{theorem}\label{thm:main} 
Fix any norm\footnote{Here, no smoothness or strict convexity is required of $\|\cdot\|_{\R^2}$.} $\|\cdot\|_{\R^2}$ on $\R^2$.
 Then there exists $\rho_0>0$ such that \[(\rho_0^{-1},\infty)\subset \mathcal L_2.\] In other words, for every $\rho\in(0,\rho_0)$, there is $x\in\R^2\setminus\Q^2$ satisfying
\[
\liminf_{q\to\infty}q^{1/2}\|qx\|_{\T^2}=\rho.
\]
\end{theorem}

Very recently, Kleinbock~\cite{Kleinbock} proved, for arbitrary norms and $d\ge2$, that $\overline{\mathcal L_d}$ is a closed ray, in stark contrast with the classical spectrum, which has three distinct regimes.  Our result leaves open the possibility that the maximal ray in $\mathcal{L}_2$ is preceded by a dense part, which itself could potentially be an alternating union of countable and uncountable intervals\footnote{with respect to the induced linear order}.\footnote{For the Euclidean norm in dimension two, \cite[footnote on p. 173]{Cassels} combined with Kleinbock's result shows that the endpoint of the ray is not attainable.}

\subsection{Prior results on higher dimensional spectra}
Akhunzhanov and Moshchevitin~\cite{AkhunzhanovMoshchevitin} and Akhunzhanov~\cite{Akhunzhanov} gave partial information about higher-dimensional Lagrange spectra by realizing \emph{best approximation constants}\footnote{A name some authors use to refer to the reciprocal of a \emph{Lagrange constant}.}
\[
\liminf_{q\to\infty}q^{1/d}\|qx\|_{\T^d}
\]
in multiplicative windows near zero whose relative widths tend to zero.  Further progress appeared very recently in the work of Kleinbock and Ward, but before describing these we should mention substantial related progress that was made on the \emph{Dirichlet spectrum}.  

In simultaneous approximation in dimension $d$, the Dirichlet constant is
\[
\limsup_{Q\to\infty}Q^{1/d}
\min_{1\le q\le Q}\|qx\|_{\T^d}.
\]
It measures uniform approximation,
whereas the Lagrange spectrum is defined by the 
liminf of best approximation constants.
Akhunzhanov and Shatskov~\cite{AkhunzhanovShatskov} proved that the Dirichlet spectrum is an interval for the Euclidean norm in dimension two, and Schleischitz~\cite{Schleischitz}
established this for the supremum norm in every dimension at least two.
Further developments appear in~\cite{KleinbockRao,Agin,HussainSchleischitzWard}.
Agin and Weiss~\cite{AginWeiss} established the interval property for arbitrary norms and more generally for $m\times n$ matrices with $\max(m,n)>1$, and proved that every value is attained on a dense uncountable set of matrices.
Thus the interval structure of the higher-dimensional Dirichlet spectrum is well-understood, while the corresponding question for the Lagrange spectrum remains open.

Kleinbock~\cite{Kleinbock} proved recently that $\overline{\mathcal L_d}$ is a closed ray for arbitrary norms and $d\ge2$. His argument proceeds through the full dynamical Lagrange spectrum, which allows arbitrary initial lattices. In particular, his Theorem~1.3 implies that, for the flow considered here, the set
\[
\left\{
\limsup_{t\to\infty}\frac{1}{\sys(g_t\Lambda)}
:\Lambda\in X_3
\right\}
\]
is a closed ray, where $\sys$ is the shortest-vector length.
We note that the Lagrange spectrum is related to \[
\left\{
\limsup_{t\to\infty}\frac{1}{\sys(g_th_x\Z^3)}
:x\in \R^2
\right\}
\] which is a restriction to special initial lattices
$h_x\Z^3$. 
Kleinbock showed density of the Lagrange spectrum in a closed ray, but does not establish uncountability of the spectrum or the existence of a Hall ray. 

In another recent preprint, Ward~\cite{Ward} obtained quantitative information about the vectors realizing small best approximation constants.
For the Euclidean norm, he showed there is a fixed $\delta>0$ such that,
for every sufficiently small $\varepsilon>0$, the vectors whose constants lie in
\[
[\varepsilon,\varepsilon(1+\delta\varepsilon^d)]
\]
form a set of positive Hausdorff dimension.
For suitably wider windows, these dimensions tend to $d$ as
$\varepsilon\to0$.
The existence of a Hall ray for square-matrix approximation was obtained by Ward~\cite[Theorem~7.1]{Ward} with respect to the supremum norms, corresponding to the $(n,n)$-flow
\[
\operatorname{diag}(e^tI_n,e^{-t}I_n).
\]
His argument uses the diagonal embedding $x\mapsto xI_n$, which preserves the best approximation constant and thus embeds the classical one-dimensional Hall ray into the square-matrix spectrum. 
For the $(2,1)$-flow, addressed by Theorem~\ref{thm:proper-ray} below, no such reduction argument is available.  

\begin{remark}
The proof of our main result is independent of the results obtained in \cite{Kleinbock} and \cite{Ward}.  
\end{remark} 

\subsection{Dynamical reformulation and idea of proof}
In dimension one, continued fractions give a discrete dynamical description of the Lagrange spectrum. Let $\sigma$ be the left shift on $\N^\Z$ and define
\[
f((a_j)_{j\in\Z})
=[a_0;a_1,a_2,\ldots]+[0;a_{-1},a_{-2},\ldots].
\]
Perron's formula gives
\[
\mathcal L_1
=
\left\{
\limsup_{n\to\infty}f(\sigma^n a):a\in\N^\Z
\right\};
\]
see~\cite[p.~146]{Moreira}.
The shift is the symbolic form of the natural extension of the Gauss map,
and the two terms in $f$ depend separately on the future and the past.
This description connects the spectrum to sums of continued-fraction Cantor sets and gives the classical construction of a Hall ray.
A comparably explicit symbolic description of the Lagrange spectrum for simultaneous approximation in dimension $d\ge2$ is not known in general. 

In this paper, despite the lack of an analogue for Perron's formula, we construct a symbolic return system for the continuous diagonal flow on $X_3$. 
We use the dynamical interpretation of Lagrange spectrum via the Dani correspondence. Denote the space of unimodular lattices by $X_3=\SL_3(\R)/\SL_3(\Z).$
Consider the diagonal flow
$
g_t=\operatorname{diag}(e^{-2t},e^t,e^t)
$ acting by left multiplication on $X_3$.

For $x\in\R^2$, let
$$h_x=\begin{pmatrix}1&0\\-x&I_2\end{pmatrix}.$$
Dani's correspondence~\cite{Dani} relates $\liminf q^{1/2}\|qx\|_{\T^2}$ to the liminf of the systole of the orbit $g_th_x\Z^3$. We thus reduce Theorem \ref{thm:main} to establishing a continuous ray for the systole of the orbit.

More precisely, given a \emph{height function}, i.e. a continuous proper $F:X_3\to[0,\infty)$, we define 
\begin{equation}
    \mathcal{L}_F(3) := \left\{ \limsup_{t\to\infty} F(g_th_x\Z^3) : x \in \R^2 \right\}.
\end{equation}

In Section~2 we deduce Theorem \ref{thm:main} from the following using $F(\Lambda)=1/\sys(\Lambda)$.  

\begin{theorem}\label{thm:proper-ray}
For any height function $F$, $\mathcal{L}_F(3)$ contains the ray $(T,\infty)$ for some $T>0$.
\end{theorem}

To prove Theorem~\ref{thm:proper-ray}, we need to construct families of bounded forward orbits  $\{g_th_x\Z^3:t\ge0\}$. A natural starting point would be periodic orbits, but the flow $g_t$ has no periodic orbits on $X_3$.  One approach is to exploit the cross sections defined by Cheung and Chevallier in \cite{CheungChevallier} or Shapira and Weiss in \cite{ShapiraWeiss} and try to construct bounded orbits using them; however, the first return maps are difficult to compute.  The key observation in the proof of Theorem~\ref{thm:proper-ray} is a \emph{generalized Poincare section}, whose base dynamics form a Smale horseshoe, with the fiber dynamics governed by an $\SL_2(\R)$-valued cocycle. The horseshoe supplies numerous periodic orbits in the base, and selecting finite words with elliptic fiber returns gives bounded lifted orbits whose closures are invariant two-tori. (See Theorem~\ref{thm:invariant-two-tori}.)

From there, we consider Hausdorff limits of those orbit closures, which form a continuous family of $g_t$-invariant compact subsets of $X_3$ (see Theorem~\ref{thm:invariant-three-tori}) and achieve a continuous ray for the height function $F$. Finally, we construct actual orbits given by concatenations of repeated words that achieve the same heights. 

Dynamically, Theorem~\ref{thm:proper-ray} gives a Hall ray for every height function along the initial lattices $h_x\Z^3$.  A key property of $\SL_2\R$ that is needed is that the set of elliptic elements form an open set.  
The same method for Theorem~\ref{thm:proper-ray} should yield an analogous result for the $(2,2)$-flow
$\operatorname{diag}(e^t,e^t,e^{-t},e^{-t})$
on $\SL_4(\R)/\SL_4(\Z)$ with respect to \emph{arbitrary} norms and \emph{height} functions.


\subsection{Related results and open questions}
The Hall-ray phenomenon extends to other rank-one geometric settings.
Parkkonen and Paulin~\cite{ParkkonenPaulin} obtained Hall-ray results
for penetration spectra of geodesics in negative curvature.
Artigiani, Marchese, and Ulcigrai~\cite{ArtigianiMarcheseUlcigrai}
proved that the Lagrange spectrum of every Veech surface contains a Hall ray. This gives an analogue associated with closed $\SL_2(\R)$-orbits in moduli spaces of translation surfaces.
They also established Hall rays for cusp spectra on finite-area hyperbolic surfaces and their persistence under sufficiently small Lipschitz perturbations of cusp height functions~\cite{PersistentHall}.

Hall rays also occur in one-dimensional inhomogeneous approximation, where the irrational slope is fixed and the shift varies. Crisp, Moran, and Pollington~\cite[Theorem~1]{CrispMoranPollington} proved that, for every $\alpha\in\R\setminus\Q$, there exists $c_\alpha>0$ such that
\[
[0,c_\alpha]\subset
\left\{
\liminf_{q\to\infty}q\min_{p\in\Z}|q\alpha-\beta-p|:\beta\in\R
\right\}.
\]
Thus the corresponding reciprocal spectrum contains a Hall ray.
Their proof uses Davenport expansions and Hall's theorem on Cantor sets.
When $\alpha$ has unbounded partial quotients in its ordinary continued fraction expansion, they further show that the entire one-sided inhomogeneous spectrum displayed above is an interval~\cite[Theorem~4]{CrispMoranPollington}.

An immediate next goal is to prove the existence of a Hall ray in every dimension. The method in this paper uses the fact that elliptic matrices form an open subset of $\mathrm{SL}_2\R$ to construct orbits whose fiber dynamics is elliptic. This cannot directly generalize to higher dimension.  It is also natural to consider the problem of the existence of a Hall ray when the parameter $x$ in Theorem~\ref{thm:main} is restricted to a non degenerate one-dimensional subspace of $\R^2$.  

Another question is to understand the level sets of the spectrum. The Lagrange spectrum describes which values occur, but not how many vectors realize a given value. For $t\in\mathcal L_d$, we may also define
\[
 D_d(t)
 =
 \dim_H\{\theta\in\R^d:\liminf_{q\to\infty}q^{1/d}\|q\theta\|_{\T^d}=1/t\}.
\]

The construction in this paper (see Proposition \ref{prop:uncountable-levels}) shows that the level sets $\{\theta\in\R^d:\liminf_{q\to\infty}q^{1/d}\|q\theta\|_{\T^d}=1/t\}$ are uncountable for $t\in (\rho_0^{-1},\infty)$ and $d=2$. It would be an interesting problem to determine the function $D_d(t)$ on the Hall ray. In particular, one may ask for effective lower bounds and determine its asymptotic behavior as $t\to\infty$. Ward \cite{Ward} showed that the Hausdorff dimension of the union of level sets tends to $d$ for suitable windows as $t\to \infty$. One may further ask whether the actual level sets also have Hausdorff dimension $D_d(t)\to d$ as $t\to \infty$.

\subsection{Outline and sketch of proof}

The construction has four steps.
\begin{enumerate}
\item \textbf{A symbolic description of actual lattice orbits.}
Section~\ref{sec:return-system} constructs a virtual transverse section using two integral changes of marking.
In its base coordinates, the two returns form a horseshoe: the future determines an expanding coordinate and the past determines a contracting coordinate.
The remaining fiber coordinate is a matrix in $\SL_2(\R)$.
Thus a word determines the base part of an orbit, while a matrix product records whether its fiber part stays bounded.
The base horseshoe is identified with $\Omega=\{\mathsf a,\mathsf b\}^{\Z}$, and the return map sends $C(z,E)$ to $C(\sigma z,EM(z))$ for $z\in \Omega$.

\item \textbf{Bounded orbits from long periodic words.}
Section~\ref{sec:periodic} begins with the constant word $\mathsf a^\infty$ and shows that its fiber return is conjugate to an irrational rotation by $\omega_{\mathsf a}$.
We then consider the orbit given by the word $(\mathsf a^n\mathsf b)^\infty$, which is a periodic orbit in the base $\Omega$. We identify an open proper subset $ I_{\mathrm{ell}}\subset \mathbb{T}^1$ and show that along any sequence \(n_i\to\infty\) with \(n_i\omega_{\mathsf a}\to\theta\in I_{\mathrm{ell}}\), the induced fiber returns are eventually elliptic and conjugate to an irrational rotation. This shows that the orbit closure of the word $(\mathsf a^{n_i}\mathsf b)^\infty$ is bounded in $X_3$.

\item \textbf{A continuous interval of limiting depths.}
Section~\ref{sec:depths} identifies the orbit closures $\mathcal O_n$ as invariant two-tori and studies their limits as $n$ grows. 

If $n_i\omega_{\mathsf a}\to \theta \in I_{\mathrm{ell}}$, the orbit closures also converge to a compact set $\mathcal O_{n_i}\to \mathcal K(\theta)\subset X_3$ in Hausdorff distance. The limit set  $\mathcal K(\theta)$ can be described purely in terms of $\theta$, and consists of three parts, one of which is an immersed invariant three-torus.
We show that the sets $\mathcal K(\theta)$ vary continuously with respect to $\theta$, and the minimum systoles tend to zero as $\theta$ approaches an endpoint of $I_{\mathrm{ell}}$, so continuity produces an interval of heights achieved by $\mathcal K(\theta)$. In particular, the heights of the actual orbits $\mathcal O_n$ are dense in that interval.

\item \textbf{One orbit with the exact prescribed lower limit.}
In Section~\ref{sec:realization}, we consider  words of the form $$\cdots (\mathsf a^{n_1}\mathsf b)^{L_1}(\mathsf a^{n_2}\mathsf b)^{L_2}(\mathsf a^{n_3}\mathsf b)^{L_3}\cdots$$  to obtain orbits that realize any given height in $(T,\infty)$.
Each repetition is chosen to make its fiber return close to the identity, so the accumulated fiber matrix is in a neighborhood where the desired heights remain attainable.
\end{enumerate}

\subsection{Notation}
We now fix some notation. 
\begin{itemize}
    \item Let $\N_0=\{0,1,2,\ldots\}$. 
    \item Let $\T^1=\R/(2\pi\Z)$. 
    \item For $\theta\in \mathbb T^1$, let
\[
R_\theta=\begin{pmatrix}\cos\theta&-\sin\theta\\ \sin\theta&\cos\theta\end{pmatrix}.
\]
\item  Let $\|\cdot\|_2$ denote Euclidean norm in $\R^2$. Matrix norms in this paper are Euclidean operator norms.
\item For $A\in\SL_2(\R)$, \emph{elliptic}, \emph{parabolic}, and \emph{hyperbolic} mean, respectively,
\[
|\tr A|<2,\qquad |\tr A|=2\text{ and }A\ne\pm I_2,\qquad |\tr A|>2.
\] 
\item Fix a Riemannian metric on $X_3$. Let $d_H$ denote Hausdorff distance between nonempty compact sets in $X_3$.
\end{itemize}

\section*{Acknowledgments}

The authors made substantial use of generative AI to assist with the research, both in the discovery of arguments and the polishing of the exposition.  All AI-generated statements and proofs were carefully checked and revised by the authors, who take full responsibility for their correctness.

\section{Dani correspondence}\label{sec:dani}

On $\R^3$, we consider the norm  \[\|(r,z)\|_{\R^3}:=\max\{|r|,\|z\|_{\R^2}\}.\] For a lattice $\Lambda\subset \R^3$ its systole is given by \[\mathrm{sys}(\Lambda):=\min_{w\in\Lambda\setminus\{0\}}\|w\|_{\R^3}.\]

The lattice $h_x\Z^3$ associated with $x\in \R^2$ consists of vectors $(q,p-qx)$.
Under the flow $g_t$, the first coordinate becomes smaller and the error coordinates become larger.
The smallest length attained by one such vector is therefore determined by the product $|q|^{1/2}\|p-qx\|_{\R^2}$. This gives the following form of Dani's correspondence,
which we note is a specialization of~\cite[Theorem 1.2]{Kleinbock} to our normalization.
We include the proof following the argument of Dani~\cite{Dani} for the convenience of the reader.

\begin{lemma}\label{lem:dani}
For every norm $\|\cdot\|_{\R^2}$ on $\R^2$ and every $x\in\R^2$,
\begin{equation}\label{eq:dani}
\liminf_{t\to\infty}\sys(g_th_x\Z^3)
=\left(\liminf_{q\to\infty}q^{1/2}\|qx\|_{\T^2}\right)^{2/3}.
\end{equation}
\end{lemma}

\begin{proof}
If $x\in\Q^2$, choose $q_0\in\N$ such that $q_0x\in\Z^2$.
Then $\|kq_0x\|_{\T^2}=0$ for every $k\in\N$, while
\[
0<\sys(g_th_x\Z^3)\le q_0e^{-2t}\longrightarrow0.
\]
Thus both sides of~\eqref{eq:dani} vanish.
Henceforth assume $x\notin\Q^2$.

For $q\in\Z\setminus\{0\}$ and $p\in\Z^2$, write $d(q,p)=\|qx-p\|_{\R^2}>0$.
Balancing the two terms in the maximum gives
\begin{equation}\label{eq:direct-min}
\begin{split}
\min_{t\in\R}\|g_th_x(q,p)\|_{\R^3}
&=\min_{t\in\R}\max\{e^{-2t}|q|,e^t d(q,p)\}\\
&=\bigl(|q|^{1/2}d(q,p)\bigr)^{2/3}.
\end{split}
\end{equation}
The minimum occurs when $e^{3t}=|q|/d(q,p)$.
For the upper bound in~\eqref{eq:dani}, choose $q\to\infty$ along a sequence realizing the approximation lower limit and choose a nearest integer vector $p$ with respect to $\|\cdot\|_{\R^2}$.
The quantities $d(q,p)=\|qx\|_{\T^2}$ are bounded by the covering radius of $\Z^2$ for $\|\cdot\|_{\R^2}$, so the minimizing times tend to infinity.

By Minkowski’s convex body theorem, the systole is uniformly bounded above on \(X_3\).
Vectors with $q=0$ have length at least $e^t\min_{p\in\Z^2\setminus\{0\}}\|p\|_{\R^2}$.
For every fixed $Q$, vectors with $0<|q|\le Q$ have lengths tending uniformly to infinity, because
$\min_{1\le q\le Q}\|qx\|_{\T^2}>0$.
Thus shortest vectors have $|q|\to\infty$ as $t\to\infty$.
Their lengths are bounded below by~\eqref{eq:direct-min}, which proves the reverse inequality.

\end{proof}

This gives a proof of Theorem \ref{thm:main}, given Theorem \ref{thm:proper-ray}.

 \begin{proof}[Proof of Theorem~\ref{thm:main} given Theorem~\ref{thm:proper-ray}]
Take $
F(\Lambda)=\sys(\Lambda)^{-1}$ which is positive and continuous.
By Mahler's compactness criterion, for every
$R>0$,
\[
\{\Lambda\in X_3:F(\Lambda)\le R\}
=
\{\Lambda\in X_3:\sys(\Lambda)\ge R^{-1}\}
\]
is compact. Thus $F$ is continuous and proper.

Let $T>0$ be given by Theorem~\ref{thm:proper-ray}, and put
$\rho_0=T^{-3/2}$. For every $0<\rho<\rho_0$, we have
$\rho^{-2/3}>T$, so there exists $x\in\R^2$ such that
\[
\limsup_{t\to\infty}F(g_th_x\Z^3)=\rho^{-2/3}.
\]

We may therefore apply~\eqref{eq:dani} to obtain
\[
\left(\liminf_{q\to\infty}q^{1/2}\|qx\|_{\T^2}\right)^{2/3}
=\liminf_{t\to\infty}\sys(g_th_x\Z^3)
=\left(\limsup_{t\to\infty}F(g_th_x\Z^3)\right)^{-1}
=\rho^{2/3}.
\]
Taking the $3/2$ power proves the assertion.
\end{proof}

\section{The return map and symbolic coding}\label{sec:return-system}\label{sec:section}

In this section, we define a virtual section and a return map that factors over a standard shift.

\subsection{Transversals, Poincare sections and return maps}
Let $\mathcal{S}$ be a compact subset of a space carrying an $\R$-action $(g_t)_{t\in\R}$.  We call $\mathcal{S}$ a \emph{transversal}, or a \emph{section}, if for each $x$ in $\mathcal{S}$ there is a $\delta>0$ such that $g_tx\not\in\mathcal{S}$ for all $0<|t|<\delta$.  It is \emph{forward invariant}, or a \emph{Poincare section}, if the set of \emph{positive return times} 
  $$\tau^+_\mathcal{S}(x):=\{~ t>0 ~:~ g_tx\in \mathcal{S} ~\} $$
is nonempty for \emph{every} $x$ in $\mathcal{S}$.  
The \emph{first return map} $\mathcal{F}_1:\mathcal{S}\to\mathcal{S}$ is defined by $$\mathcal{F}_1(x):= g_{\tau_1(x)}x$$ where $\tau_1(x):=\min \tau^+_\mathcal{S}(x)$ is the \emph{first return time}.  

The first return map is usually difficult to compute and more flexibility is gained by considering a larger class of return maps.  
By a \emph{positive return map} we mean a function $\mathcal{F}:\mathcal{S}\to\mathcal{S}$ with the property that for each $x$ in $\mathcal{S}$ there exists $t>0$ such that $\mathcal{F}(x)=g_tx$.  
Every positive integer-valued function $n:\mathcal{S}\to\Z_+$ determines a positive return map via 
$$ \mathcal{F}_n(x) = \mathcal{F}_1^{n(x)}(x) $$
and every such map arises in this manner.

In this section, we construct a subset $\tilde\Sigma\subset \mathrm{SL_3\R}$ in the space of marked frames, with a local coordinates parametrization $$\Tilde\Sigma \to \R^2\times \R^2\times \mathrm{SL_2\R}.$$ We call the $\R^2\times \R^2$ part the \emph{base} and the $\mathrm{SL_2\R}$ part the \emph{fiber}. 

Consider a compact subset of the virtual section $\widetilde\Sigma$ projected to $\mathrm{SL}_3\R/\mathrm{SL}_3\Z$ where the first return map is well defined.  Via Proposition~3.1 below we shall identify a positive return map (not necessarily first) of $\mathcal{S}\subset \tilde\Sigma$ that factors over a horseshoe $\mathcal{H}=\bigcap U_n \times \bigcap V_n$ where the sets $U_n$ are determined by a pair of contracting maps $\Phi_{\mathsf a},\Phi_{\mathsf b}:U\to U$ with disjoint images by setting $U_0=U$ and $U_{n+1}=\Phi_{\mathsf a}(U_n)\cup\Phi_{\mathsf b}(U_n)$, and the sets $V_n$ are similarly determined by a pair $\Psi_{\mathsf a},\Psi_{\mathsf b}:V\to V$.  Here, $U$ and $V$ are compact subsets of the unstable and stable subspaces, respectively. This gives a symbolic encoding of the system restricted to the base. The fiber coordinates will in turn become a cocycle over the base.

\subsection{The virtual section $\Tilde\Sigma$}

A marked frame is a matrix $g\in\SL_3(\R)$, with associated lattice $g\Z^3$.
Write $g=\left(\begin{smallmatrix}a&b\\c&D\end{smallmatrix}\right)$ in $1+2$ block form.
On the set $a>0$, $\det D>0$, put
\[
H(g)=\log\frac{\det D}{a^2}.
\]
The flow rescales the first row by $e^{-2t}$ and the last two rows by $e^t$, so
\begin{equation*}
H(g_tg)=H(g)+6t.
\end{equation*}
Thus
\begin{equation*}
\widetilde\Sigma=\{g:H(g)=0\}
\end{equation*}
is transverse to the marked-frame flow.
Its coordinates are
\begin{equation*}
v=b/a,\qquad u=-D^{-1}c,\qquad E=D/a\in\SL_2(\R).
\end{equation*}
Conversely, for a column $u$, a row $v$, and $E\in\SL_2(\R)$ with $1+vu>0$, set
\begin{equation*}
C(u,v,E)=(1+vu)^{-1/3}\begin{pmatrix}
    1&0\\0&E
\end{pmatrix}\cdot
\begin{pmatrix}1&v\\-u&I_2\end{pmatrix}.
\end{equation*}
The Schur complement gives $\det C(u,v,E)=1$ and shows that this parametrizes $\widetilde\Sigma$ uniquely.
In particular,
\begin{equation*}
C(u,v,E)=\operatorname{diag}(1,E)C(u,v,I_2).
\end{equation*}
The image of the section in $X_3$ can have self-intersections because different integral markings represent the same lattice.
We work with prescribed changes of marking given below.

\subsection{Two integral matrices and their return maps}

Let
\begin{equation*}
U=[-1/5,1/5]\times[-1/50,1/50],\qquad
V=[-1/50,1/50]\times[-1/5,1/5],
\end{equation*}
with elements of $U$ viewed as columns and elements of $V$ as rows.
Choose the integral changes of marking
\begin{equation}\label{eq:Gamma-ab}
\Gamma_{\mathsf a}=\begin{pmatrix}101&-1&-10\\10&0&-1\\1&0&0\end{pmatrix},
\qquad
\Gamma_{\mathsf b}=\begin{pmatrix}101&-1&10\\10&0&1\\-1&0&0\end{pmatrix}
\in\SL_3(\Z).
\end{equation}

We compute the positive return maps that exactly change the marking from $\tilde\Sigma$ to $\tilde\Sigma\Gamma_s^{-1}, s\in\{\mathsf a,\mathsf b\}$ as follows.

\begin{proposition}\label{prop:return}
For $s\in\{\mathsf a,\mathsf b\}$, and $E\in\SL_2(\R)$, there exist unique maps $\Phi_s: U\to U$, $\Psi_s:V\to V$, $M_s: U\to \mathrm{SL}_2\R$ and a function $h_s:U\times V\to (1,\infty)$, such that
\begin{equation}\label{eq:exact-marked-return}
g_{h_s(u,v)}C(\Phi_s(u),v,E)\Gamma_s
=C\bigl(u,\Psi_s(v),EM_s(u)\bigr)
\end{equation}

Consequently,
\begin{equation*}
g_{h_s(u,v)}C(\Phi_s(u),v,E)\Z^3
=C\bigl(u,\Psi_s(v),EM_s(u)\bigr)\Z^3.
\end{equation*}
\end{proposition}

\begin{proof}
Since $g_t$ commutes with $\begin{pmatrix}
    1&0\\0&E
\end{pmatrix}$ for all $E$, we only need to prove the Proposition for $E=I_2$. 

Write $\Gamma_s=\left(\begin{smallmatrix}a_s&b_s\\c_s&D_s\end{smallmatrix}\right)$.
For a matrix $\begin{pmatrix}1&y\\-u&I_2\end{pmatrix}$, put
$
Q(u,y)=(1+yu)I_2-uy.$
Since $1+yu\ge1-1/125>0$ on $U\times V$, the exact inverse is
\[
\begin{pmatrix}1&y\\-u&I_2\end{pmatrix}^{-1}
=\frac1{1+yu}
\begin{pmatrix}1&-yQ(u,y)\\u&Q(u,y)\end{pmatrix}.
\]
Indeed, $yQ(u,y)=y$ and $\det Q(u,y)=1+yu>0$.

Fix $u\in U$ and $v\in V$, and seek $w\in U$ and $y\in V$ such that
$C(w,v,I_2)\Gamma_sC(u,y,I_2)^{-1}$ is block diagonal.
The scalar normalizations in the charts do not affect the off-diagonal blocks, so this is equivalent to requiring
\[
\begin{pmatrix}1&v\\-w&I_2\end{pmatrix}
\begin{pmatrix}a_s&b_s\\c_s&D_s\end{pmatrix}
\begin{pmatrix}1&-yQ(u,y)\\u&Q(u,y)\end{pmatrix}
\]
to be block diagonal.
Its lower-left and upper-right blocks are, respectively,
\[
c_s+D_su-w(a_s+b_su),
\qquad \bigl(b_s+vD_s-(a_s+vc_s)y\bigr)Q(u,y).
\]
Both denominators $a_s+b_su$ and $a_s+vc_s$ are $\ge 100$ on the given rectangles.
Since $Q(u,y)$ is invertible, the two blocks vanish if and only if
$w=\Phi_s(u)$ and $y=\Psi_s(v)$, where
\begin{equation*}
\Phi_s(u)=\frac{c_s+D_su}{a_s+b_su},\qquad
\Psi_s(v)=\frac{b_s+vD_s}{a_s+vc_s}.
\end{equation*}

For these choices, put $F_s(u)=D_s-\Phi_s(u)b_s$.
Then
\[
C(\Phi_s(u),v,I_2)\Gamma_s
=\left(\frac{1+\Psi_s(v)u}{1+v\Phi_s(u)}\right)^{1/3}\operatorname{diag}(a_s+vc_s,F_s(u))C(u,\Psi_s(v),I_2).
\]
The Schur complement, applied to
$\Gamma_s\left(\begin{smallmatrix}1&0\\u&I_2\end{smallmatrix}\right)$, gives
\[
\det F_s(u)=\frac1{a_s+b_su}>0.
\]
After restoring the chart normalizations, the upper-left entry determines the return time uniquely by
\[
e^{2h_s(u,v)}
=(a_s+vc_s)
\left(\frac{1+\Psi_s(v)u}{1+v\Phi_s(u)}\right)^{1/3}.
\]
The lower block then determines
$M_s(u)=F_s(u)/\sqrt{\det F_s(u)}\in\SL_2(\R)$.
These formulas give exactly $$g_{h_s(u,v)}C(\Phi_s(u),v,I_2)\Gamma_sC(u,\Psi_s(v),I_2)^{-1}=\mathrm{diag}(1,M_s(u))$$ which implies Equation ~\eqref{eq:exact-marked-return}.

In fact, for the matrices in~\eqref{eq:Gamma-ab}, these maps are
\begin{equation}\label{eq:PhiPsi}
\begin{aligned}
\Phi_{\mathsf a}(u)&=\frac{(10-u_2,1)^T}{101-u_1-10u_2},
&\Phi_{\mathsf b}(u)&=\frac{(10+u_2,-1)^T}{101-u_1+10u_2},\\
\Psi_{\mathsf a}(v)&=\frac{(-1,-10-v_1)}{101+10v_1+v_2},
&\Psi_{\mathsf b}(v)&=\frac{(-1,10+v_1)}{101+10v_1-v_2}.
\end{aligned}
\end{equation}

Let
\begin{equation}\label{eq:lambda}
\begin{gathered}
d_{\mathsf a}(u)=101-u_1-10u_2,\qquad d_{\mathsf b}(u)=101-u_1+10u_2,\\
\lambda_s(u,v)=d_s(u)
\left(\frac{1+v\Phi_s(u)}{1+\Psi_s(v)u}\right)^{2/3}.
\end{gathered}
\end{equation}

and 
\begin{equation}\label{eq:explicit-cocycle}
\begin{aligned}
h_s(u,v)&=\tfrac12\log\lambda_s(u,v),\\
M_{\mathsf a}(u)&=\frac1{\sqrt{d_{\mathsf a}(u)}}
\begin{pmatrix}10-u_2&u_1-1\\1&10\end{pmatrix},\\
M_{\mathsf b}(u)&=\frac1{\sqrt{d_{\mathsf b}(u)}}
\begin{pmatrix}10+u_2&1-u_1\\-1&10\end{pmatrix}.
\end{aligned}
\end{equation}

Direct substitution gives $\Phi_s(U)\subset U$, $\Psi_s(V)\subset V$, and $\lambda_s>99$, so $h_s>1$.
\end{proof}

We now show that the base maps $\Phi_s$ and $\Psi_s$ are contractions.

\begin{lemma}\label{lem:branches}
For $s\in\{\mathsf a,\mathsf b\}$, the maps $\Phi_s:U\to U$ and $\Psi_s:V\to V$ are injective contractions satisfying
\begin{equation*}
\begin{aligned}
\|\Phi_s(u)-\Phi_s(u')\|_\infty&\le900^{-1}\|u-u'\|_\infty,\\
\|\Psi_s(v)-\Psi_s(v')\|_\infty&\le900^{-1}\|v-v'\|_\infty.
\end{aligned}
\end{equation*}
The images $\Phi_{\mathsf a}(U)$ and $\Phi_{\mathsf b}(U)$ are disjoint, as are $\Psi_{\mathsf a}(V)$ and $\Psi_{\mathsf b}(V)$.
\end{lemma}

\begin{proof}
All denominators in~\eqref{eq:PhiPsi} are at least $100.6$.
Without loss of generality, we take $\Phi_{\mathsf a}$ for example; the computations for the other three maps are similar. Taking the difference $\Phi_{\mathsf a}(u) - \Phi_{\mathsf a}(u')$, we obtain that
\begin{equation*}
    \begin{aligned}
    &\Phi_{\mathsf a}(u) - \Phi_{\mathsf a}(u') \\
    &= \frac{(-(u_2-u_2') + 10(u_1-u_1')+u_2(u_1'-u_1) + (u_2-u_2')u_1, (u_1-u_1') + 10(u_2-u_2'))^T}{(101-u_1 -10u_2) (101-u_1' -10u_2')}.
    \end{aligned}
\end{equation*}
Taking the absolute value of the coordinates, we have
\begin{equation*}
\begin{aligned}
\left|(\Phi_{\mathsf a}(u) - \Phi_{\mathsf a}(u'))_1\right| &\le \frac{1+10+0.02 +0.2}{100.6^2}\|u-u'\|_\infty, \\
\left|(\Phi_{\mathsf a}(u) - \Phi_{\mathsf a}(u'))_2\right| &\le \frac{1+10}{100.6^2}\|u-u'\|_\infty. \\
\end{aligned}
\end{equation*}
Therefore, 
\[
\| \Phi_{\mathsf a}(u) - \Phi_{\mathsf a}(u') \|_\infty \le \frac{11.22} {100.6^2}\|u-u'\|_\infty \le 900^{-1}\|u-u'\|_\infty.
\]

Injectivity follows from the invertibility of the integral matrices inducing these base maps.
\end{proof}

\begin{lemma}\label{lem:LipMsu}
$M_s:U\to \mathrm{SL}_2\R$ is Lipschitz with respect to the infinity norm for $u$ and operator norm for $M_s(u)$:
$$ \|M_s(u) - M_s(u')\| \le 0.2\| u-u'\|_\infty.$$
\end{lemma}
\begin{proof} We state the proof for $s={\mathsf a}$, the case $s=\mathsf b$ is exactly the same since they only differ by signs.

 We have
    $$
    M_{\mathsf a}(u)-M_{\mathsf a}(u')
    =
    \frac{1}{\sqrt{d_{\mathsf a}(u)}}\begin{pmatrix}
       -( u_2-u_2^\prime)& u_1-u_1^\prime\\0&0
    \end{pmatrix}
    +\begin{pmatrix}
        10-u_2^\prime& u_1^\prime-1\\1&10
    \end{pmatrix}
    \left(
    \frac1{\sqrt{d_{\mathsf a}(u)}}-\frac1{\sqrt{d_{\mathsf a}(u')}}
    \right).
    $$
    We note that
    $
    \|\begin{pmatrix}
        10-u_2^\prime& u_1^\prime-1\\1&10
    \end{pmatrix}\|  \le 10.5.
    $ and $d_{\mathsf a}\in [100.6,101.4]$. Moreover, a direct computation shows $|\frac1{\sqrt{d_{\mathsf a}(u)}}-\frac1{\sqrt{d_{\mathsf a}(u')}}|\le \frac{11}{2000}\|u-u^\prime\|_\infty.$
    
    This yields
    $$
    \begin{aligned}
    \|M_{\mathsf a}(u)-M_{\mathsf a}(u')\|
    &\le 
    \frac{\sqrt{2}}{10}\|u-u^\prime\|_\infty
    + 10.5
    \cdot \frac{11} {2000} \cdot\| u- u'\|_\infty \\
    &\le 0.2 \| u- u'\|_\infty
    \end{aligned}
    $$
\end{proof}
\begin{corollary}\label{lem:singMsu}
The singular values of $M_s(u)$ can be uniformly bounded for $s\in \{\mathsf a,\mathsf b\}$ and $u\in U$:
$$ 0.96\le \sigma_i(M_s(u))\le 1.04, \quad i = 1, 2. $$
\end{corollary}
\begin{proof}
Note that at $u=(0,0)$, the matrix $M_s(0)$ is orthogonal, hence
$$\sigma_1(M_s(0))=\sigma_2(M_s(0))=1.$$
Since singular values are $1$-Lipschitz with respect to the operator norm,
$$ |\sigma_i(M_s(u))-1| \le \|M_s(u)-M_s(0)\|. $$
Using the Lipschitz bound from Lemma~\ref{lem:LipMsu} and the fact that \(\|u\|_\infty\le 1/5\), we have
$$ \|M_s(u)-M_s(0)\| \le 0.2 / 5=0.04. $$
Hence
$$ 0.96\le \sigma_i(M_s(u))\le 1.04 $$
uniformly.
\end{proof}

Thus the current coordinate $\Phi_s(u)$ identifies the branch, and the forward return sends it to $u$.
This expands the unstable coordinate while contracting the stable coordinate by $\Psi_s$.
The return time is independent of $E$, and $E$ is multiplied on the right by $M_s$.

\subsection{The horseshoe and its symbolic coding}\label{sec:words}

We take the set of bi-infinite sequences with the left shift
\[
\Omega=\{\mathsf a,\mathsf b\}^{\Z},\qquad
(\sigma z)(i)=z(i+1).
\] We now show that the maximal invariant set of $\Phi,\Psi$ is a Smale horseshoe, which has a standard symbolic coding.
\begin{proposition}
\label{prop:suspension-factor}
There exists a continuous map $\Omega\to U\times V; z\mapsto (u(z),v(z))$
whose image we denote by $\mathcal H$,
given by
\[
\begin{aligned}
u(z) &= \lim_{k \to \infty} \Phi_{z(0)} \circ \Phi_{z(1)}\circ \cdots \circ \Phi_{z(k)} (0,0), \\
v(z) &= \lim_{k \to \infty} \Psi_{z(-1)} \circ \Psi_{z(-2)}\circ \cdots \circ \Psi_{z(-k)} (0,0).
\end{aligned}
\]

 They satisfy

\begin{equation*}
u(z)=\Phi_{z(0)}\bigl(u(\sigma z)\bigr),
\qquad
v(\sigma z)=\Psi_{z(0)}\bigl(v(z)\bigr).
\end{equation*}

For $z\in\Omega$ and $E\in\SL_2(\R)$, set $$
C(z,E)=C(u(z),v(z),E),\quad
M(z)=M_{z(0)}\bigl(u(\sigma z)\bigr),
\quad
h(z)=h_{z(0)}\bigl(u(\sigma z),v(z)\bigr).$$
Then
\begin{equation}\label{eq:symbolic-marked-return}
g_{h(z)}C(z,E)\Gamma_{z(0)}
=
C(\sigma z,EM(z)).
\end{equation}

Moreover, for every $N\ge0$,
\begin{equation}\label{eq:word-closeness}
\begin{aligned}
z(i)=z'(i)\quad(0\le i<N)
&\Longrightarrow
\|u(z)-u(z')\|_\infty
\le \frac25\,900^{-N},\\
z(-i)=z'(-i)\quad(1\le i\le N)
&\Longrightarrow
\|v(z)-v(z')\|_\infty
\le \frac25\,900^{-N}.
\end{aligned}
\end{equation}

\end{proposition}

\begin{proof}
By Lemma~\ref{lem:branches}, the two return branches define a Smale horseshoe. By the symbolic coding theorem for horseshoes \cite[\S2.5]{KatokHasselblatt}, its maximal invariant set is topologically conjugate to the full two-sided shift $(\Omega,\sigma)$.  Denote the conjugacy by $
z\longmapsto (u(z),v(z)).$
Then
\[
u(z)=\Phi_{z(0)}(u(\sigma z)),
\qquad
v(\sigma z)=\Psi_{z(0)}(v(z)).
\]

Proposition~\ref{prop:return}, applied with $u=u(\sigma z)$ and $v=v(z)$, gives \eqref{eq:symbolic-marked-return}.  

If $z$ and $z'$ agree in positions $0,\ldots,N-1$, their unstable coordinates lie in
\[
\Phi_{z(0)}\circ\cdots\circ\Phi_{z(N-1)}(U),
\]
whose $\|\cdot\|_\infty$-diameter is at most $\frac25\,900^{-N}. $
The stable estimate follows in the same way from the common past.  This proves \eqref{eq:word-closeness}.

\end{proof}

As a direct application, define for any $z\in \Omega$,
\[
M^{(0)}(z)=I_2,\qquad T_0(z)=0,
\]
and, for $j\in\Z$,
\begin{equation*}
M^{(j+1)}(z)=M^{(j)}(z)M(\sigma^jz),
\qquad
T_{j+1}(z)-T_j(z)=h(\sigma^jz).
\end{equation*}

Then by inductively applying Equation \ref{eq:symbolic-marked-return}, for every $j\in\Z$ and
$0\le r\le h(\sigma^jz)$, we have
\begin{equation}\label{eq:factor-map}
g_{T_j(z)+r}C(z,E)\Z^3
=
g_rC\bigl(\sigma^jz,EM^{(j)}(z)\bigr)\Z^3.
\end{equation}
The intervals in \eqref{eq:factor-map} cover the whole flow orbit.

\subsection{Finite words and a uniform fiber estimate}\label{sec:finite-words}

If two words agree on a long sequence of symbols, it means their base coordinates are close, but a separate estimate is needed for products of fiber matrices.

By definition, $M(z)$ and $M^{(\ell)}(z)$, for $\ell\ge0$, depend only on the future $z[0,\infty)$.
For $k,\ell\in\N_0$, the cocycle identity is
\begin{equation*}
M^{(k+\ell)}(z)=M^{(k)}(z)M^{(\ell)}(\sigma^kz).
\end{equation*} 

We have the following bound on the cocycle, when the base coordinates are close.

\begin{lemma}\label{lem:fiber-bound}
For every $\ell,N\in\N_0$ and every $z,z'\in\Omega$ with
$z(i)=z'(i)$ for $0\le i<\ell+N$,
\begin{equation}\label{eq:uniform-fiber-bound}
\bigl\|M^{(\ell)}(z)^{-1}M^{(\ell)}(z')-I_2\bigr\|
\le 600^{-N}.
\end{equation}
\end{lemma}

\begin{proof}
The case $\ell=0$ is immediate. Assume $\ell\ge1$ and set
\[
u_j=u(\sigma^jz),\qquad u_j'=u(\sigma^jz')
\qquad (0\le j\le\ell).
\]
For $0\le j<\ell$, the sequences $\sigma^{j+1}z$ and
$\sigma^{j+1}z'$ agree in their first $\ell+N-j-1$ positions.
Thus~\eqref{eq:word-closeness} gives
\[
\|u_{j+1}-u_{j+1}'\|_\infty
\le \frac25\,900^{-(\ell+N-j-1)}.
\]
By definition,
\[
\begin{aligned}
M^{(\ell)}(z)
&=M_{z(0)}(u_1)\cdots M_{z(\ell-1)}(u_\ell),\\
M^{(\ell)}(z')
&=M_{z(0)}(u_1')\cdots M_{z(\ell-1)}(u_\ell').
\end{aligned}
\]
Lemma~\ref{lem:LipMsu} gives
\[
\|M_{z(j)}(u_{j+1}')-M_{z(j)}(u_{j+1})\|
\le 0.2\,\|u_{j+1}'-u_{j+1}\|_\infty,
\]
while Corollary~\ref{lem:singMsu} gives
\[
\|M_s(u)\|\le1.1,\qquad
\|M_s(u)^{-1}\|\le1.1
\qquad \forall s\in\{\mathsf a,\mathsf b\},\ u\in U.
\]

We use the telescoping sum
\[
\begin{aligned}
&M^{(\ell)}(z)^{-1}M^{(\ell)}(z')-I_2\\
&=\sum_{j=0}^{\ell-1}
\left(\prod_{k=j}^{\ell-1}M_{z(k)}(u_{k+1})\right)^{-1}
\bigl(M_{z(j)}(u_{j+1}')-M_{z(j)}(u_{j+1})\bigr)
\left(\prod_{k=j+1}^{\ell-1}M_{z(k)}(u_{k+1}')\right),
\end{aligned}
\]
where products are ordered from left to right in increasing index,
and an empty product is $I_2$.
By submultiplicativity of the operator norm,
\[
\begin{aligned}
\bigl\|M^{(\ell)}(z)^{-1}M^{(\ell)}(z')-I_2\bigr\|
&\le
\sum_{j=0}^{\ell-1}
1.1^{\,2(\ell-j)-1}
\|M_{z(j)}(u_{j+1}')-M_{z(j)}(u_{j+1})\|\\
&\le
0.2\cdot0.4\cdot
\sum_{j=0}^{\ell-1}
1.1^{\,2(\ell-j)-1}
900^{-(\ell+N-j-1)}\\
&=
0.08 \cdot1.1\cdot900^{-N}
\sum_{r=0}^{\ell-1}
\left(\frac{1.1^2}{900}\right)^r\\
&<600^{-N}.
\end{aligned}
\]
This proves~\eqref{eq:uniform-fiber-bound}.
\end{proof}

\section{The long periodic word \texorpdfstring{$\mathsf a^n\mathsf b$}{a-to-the-n b}}\label{sec:periodic}

We now consider the orbit in $X_3$ corresponding to a periodic word $\mathsf a^n\mathsf b$. The return is automatically periodic in the base, while its fiber coordinate is multiplied by a matrix $P_n$.
We compare this matrix with the constant $\mathsf a$-return, which is an irrational elliptic rotation in suitable coordinates.
The long block $\mathsf a^n$ accumulates the angle $n\omega_{\mathsf a}$, and the single $\mathsf b$, together with the change in the preceding fiber product, contributes a fixed defect matrix $B$.

\subsection{The periodic word and its period matrix}

For $n\ge1$, define $z_n\in\Omega$ by
\begin{equation*}
z_n(i)=
\begin{cases}
\mathsf b,&i\equiv n\pmod{n+1},\\
\mathsf a,&\text{otherwise}.
\end{cases}
\end{equation*}
Thus $z_n(0)=\cdots=z_n(n-1)=\mathsf a$, $z_n(n)=\mathsf b$, and $\sigma^{n+1}z_n=z_n$.

The change in the fiber and time spent for an $\mathsf a^n\mathsf b$ cycle are given by
\begin{equation*}
P_n=M^{(n+1)}(z_n),\qquad \tau_n=T_{n+1}(z_n).
\end{equation*}

We have the following description of the entire orbit.

\begin{proposition}\label{prop:complete-orbit}
For $E\in\SL_2(\R)$, integers $k\in\Z$ and $0\le j\le n$, and
$0\le r\le h(\sigma^jz_n)$,
\begin{equation}\label{eq:complete-orbit}
g_{k\tau_n+T_j(z_n)+r}C(z_n,E)\Z^3
=g_rC(\sigma^jz_n,EP_n^kM^{(j)}(z_n))\Z^3.
\end{equation}
\end{proposition}

\begin{proof}
Periodicity gives
$M^{(k(n+1)+j)}(z_n)=P_n^kM^{(j)}(z_n)$ and
$T_{k(n+1)+j}(z_n)=k\tau_n+T_j(z_n)$.
Apply Proposition~\ref{prop:suspension-factor}.
\end{proof}

Thus, if $P_n$ is elliptic, then all its powers are bounded, and this gives a bounded orbit in $X_3$. We next give an asymptotic description of $P_n$ and the intermediate matrices $M^{(j)}(z_n)$.

\subsection{The constant return and the single defect}

Let $z_{\mathsf a},z_\delta\in\Omega$ be the constant word and the single-defect word:
\begin{equation}\label{eq:zminuszdelta}
z_{\mathsf a}(i)=\mathsf a,\quad \forall i\in\Z;\qquad
z_\delta(i)=
\begin{cases}\mathsf b,&i=0,\\\mathsf a,&i\ne0.\end{cases}
\end{equation}
We shall show that the matrix cocycle in the $\mathsf a^n$-part of the $z_n$ can be approximated by $M(z_{\mathsf a})$, and the matrix for the $\mathsf b$ word is related to $M(z_\delta)$.

Before that, we first show that $M(z_{\mathsf a})$ is elliptic.

\begin{lemma}\label{lem:constant-branch}
Set $\alpha=e^{2h(z_{\mathsf a})}$.
Then $100<\alpha<101$ and
\begin{equation}\label{eq:Ma-explicit}
M(z_{\mathsf a})=\frac1{\sqrt\alpha}
\begin{pmatrix}
10-\alpha^{-1}&-1+10\alpha^{-1}-\alpha^{-2}\\
1&10
\end{pmatrix}.
\end{equation}
This matrix is elliptic.
There exist $G\in\SL_2(\R)$ and $\omega_{\mathsf a}\in(0,\pi)$, with $\omega_{\mathsf a}/\pi\notin\Q$, such that
\begin{equation}\label{eq:Ma-fixed-u}
M(z_{\mathsf a})=G R_{\omega_{\mathsf a}}G^{-1}.
\end{equation}
Note that here the matrix $G$ is only unique up to a rotation, while the operator norm of $G$ is unique, and a direct computation shows $\|G\| < 1.03$. 
\end{lemma}

\begin{proof}
For this calculation, write $\bar u=u(z_{\mathsf a})$ and $\bar v=v(z_{\mathsf a})$, the fixed points of $\Phi_{\mathsf a}$ and $\Psi_{\mathsf a}$.
Equation~\eqref{eq:lambda} gives $\alpha=d_{\mathsf a}(\bar u)=101+(-1,-10)\bar u$.
The fixed-point equation for~\eqref{eq:PhiPsi} gives
\begin{equation}\label{eq:fixed-base}
\bar u=(10\alpha^{-1}-\alpha^{-2},\alpha^{-1})^T,
\qquad
\bar v=(-\alpha^{-1},-10\alpha^{-1}+\alpha^{-2}),
\end{equation}
and
\begin{equation}\label{eq:minpoly}
\alpha^3-101\alpha^2+20\alpha-1=0.
\end{equation}
The rectangle bound gives $\alpha>100$, and the positive coordinates of $\bar u$ give $\alpha=101-\bar u_1-10\bar u_2<101$.
Substituting~\eqref{eq:fixed-base} into~\eqref{eq:explicit-cocycle} proves~\eqref{eq:Ma-explicit}. Moreover,
\[
0<\tr M(z_{\mathsf a})=\frac{20}{\sqrt\alpha}-\frac1{\alpha^{3/2}}<2.
\]
The matrix is therefore elliptic.
Its lower-left entry is positive, so the $\SL_2(\R)$ conjugacy can be chosen with rotation angle $\omega_{\mathsf a}\in(0,\pi)$.

At the fixed base coordinates, the return time is $\tfrac12\log\alpha$.
The marked return formula gives
\begin{equation}\label{eq:gammaconj}
C(z_{\mathsf a},I_2)\Gamma_{\mathsf a}C(z_{\mathsf a},I_2)^{-1}
=\operatorname{diag}\bigl(\alpha,\alpha^{-1/2}M(z_{\mathsf a})\bigr).
\end{equation}
The characteristic polynomial of $\Gamma_{\mathsf a}$ is $f(X)=X^3-101X^2+20X-1$, which has one real root $\alpha$ and two conjugate complex roots $\mu, \bar{\mu}$, and is irreducible over $\Q$ because its only possible rational roots are $\pm1$. Moreover, its discriminant $\Delta(f) = -36471<0$, which is negative. This implies
\begin{equation}\label{eq:galgrp}
\operatorname{Gal}(f/\Q) = S_3.
\end{equation}
If $\omega_{\mathsf a}/\pi$ were rational, then there would exist an integer $m>0$ such that $(M(z_{\mathsf a}))^m = GR_{m\omega_{\mathsf a}}G^{-1} = I_2$ and by taking the $m$-th power of \eqref{eq:gammaconj},
\begin{equation}\label{eq:weird}
C(z_{\mathsf a},I_2)\Gamma_{\mathsf a}^mC(z_{\mathsf a},I_2)^{-1}
=\operatorname{diag}\bigl(\alpha^m,\alpha^{-m/2},\alpha^{-m/2}\bigr).
\end{equation}
The eigenvalues of \eqref{eq:weird} satisfy 
\[
\{ \alpha^m, \mu^m, \bar{\mu}^m\} = \{\alpha^m,\alpha^{-m/2},\alpha^{-m/2}\}.
\]
The pair of complex roots $\mu,\bar\mu$ would satisfy $\mu^m=\bar\mu^m = \alpha^{-m/2}$. The automorphism given by $\alpha \mapsto \mu, \ \mu \mapsto \bar{\mu}$ from \eqref{eq:galgrp} produces the relation $\mu^{2m}=\bar\mu^{2m} = \alpha^{-m} = \alpha^{2m}$, contradicting $\alpha>100$.
\end{proof}

We fix $G$ and $\omega_{\mathsf a}$ as in~\eqref{eq:Ma-fixed-u} and $z_\delta\in\Omega$ as in~\eqref{eq:zminuszdelta}.

Define
\begin{equation}\label{eq:defect-matrix}
B=\lim_{r\to\infty}
R_{-r\omega_{\mathsf a}}G^{-1}M^{(r+1)}(\sigma^{-r}z_\delta)G.
\end{equation}
The following calculation verifies that this limit exists in $\SL_2(\R)$, with error $O(400^{-r})$.

\begin{lemma}\label{lem:nonorthogonal-defect}
The limit in Equation~\eqref{eq:defect-matrix} is well-defined, and $B\notin \SO(2)$.
\end{lemma}

\begin{proof}
    We prove by establishing Cauchy convergence. Let 
    \[
    B_r = R_{-r\omega_{\mathsf a}}G^{-1}M^{(r+1)}(\sigma^{-r}z_\delta)G.
    \]
    Then 
    \begin{equation}\label{eq:GBrGinv}
    GB_rG^{-1} = M^{(r)}(z_{\mathsf a})^{-1} M^{(r+1)}(\sigma^{-r}z_\delta) = M(z_{\mathsf a})^{-r} M^{(r+1)}(\sigma^{-r}z_\delta).
    \end{equation}
    Taking the difference between two consecutive terms:
    \begin{equation}\label{eq:GBdiffGinv}
    \begin{aligned}
        \|G(B_r - B_{r+1})G^{-1}\| = \| M^{(r)}(z_{\mathsf a})^{-1} (I - M(z_{\mathsf a})^{-1} M(\sigma^{-r-1} z_{\delta})) M^{(r+1)}(\sigma^{-r}z_\delta) \|.
    \end{aligned}
    \end{equation}
    From Lemma~\ref{lem:fiber-bound}, we control the middle term: \[\| I - M(z_{\mathsf a})^{-1} M(\sigma^{-r-1} z_{\delta})\|\le 600^{-r}. \]
    The other terms can be controlled by Corollary~\ref{lem:singMsu}:
    \[
    \begin{aligned}
        \| M^{(r)}(z_{\mathsf a})^{-1}\| \le 1.1^r, \quad \|M^{(r+1)}(\sigma^{-r}z_\delta)\| \le 1.1^{r+1}.
    \end{aligned}
    \]
    Consequently,
    \[
    \begin{aligned}
        \|G(B_r - B_{r+1})G^{-1}\| &\le 1 \times 1.1^{2r+1} \times 600^{-r}
        \le 1.2 \times 400^{-r},
    \end{aligned}
    \]
    and
    \[
    \begin{aligned}
        \|B_r - B_{r+1}\| &\le 1.2 \times 400^{-r} \|G\|^2 \le 1.3 \times 400^{-r}.
    \end{aligned}
    \]
    This shows the limit exists and in fact,
    \[\|B_r - B\| \le \sum_{n=0}^{\infty} 1.3 \times 400^{-r-n} \le 1.4 \times 400^{-r}. \]
    We also obtain from \eqref{eq:GBdiffGinv} that
    \begin{equation}\label{eq:GBlimGinv}
    \|G(B_r - B)G^{-1}\| \le \sum_{n=0}^{\infty} 1.2 \times 400^{-r-n} \le 1.3 \times 400^{-r}. 
    \end{equation}
    To prove that $B$ is not orthogonal, suppose otherwise that $B \in \SO(2)$. Then $B$ commutes with $R_{\omega_{\mathsf a}}$, and hence $GBG^{-1}$ commutes with $M(z_{\mathsf a})$. 
    Taking $r=2$ in equation \eqref{eq:GBrGinv},
    \begin{equation*}
        M(z_{\mathsf a}) GB_2G^{-1} - GB_2G^{-1} M(z_{\mathsf a}) = M(z_{\mathsf a})^{-1} M^{(3)}(\sigma^{-2}z_\delta) - M(z_{\mathsf a})^{-2} M^{(3)}(\sigma^{-2}z_\delta) M(z_{\mathsf a}),
    \end{equation*}
    where, for $\bar u=u(z_{\mathsf a})$,
    \[ M^{(3)}(\sigma^{-2}z_\delta) = M_\mathsf{a}(\Phi_{\mathsf a} (\Phi_{\mathsf b} (\bar{u})))\  M_\mathsf{a} (\Phi_{\mathsf b} (\bar{u})) \  M_\mathsf{b}(\bar{u}). \]
    Direct computation gives
    \[ M(z_{\mathsf a}) GB_2G^{-1} - GB_2G^{-1} M(z_{\mathsf a}) \approx  \begin{pmatrix}1.7417\times 10^{-8} &1.7715 \times 10^{-4} \\ 1.9664 \times 10^{-4} & -1.7417 \times 10^{-8}\end{pmatrix} \]
    and therefore 
    \begin{equation}
        \| M(z_{\mathsf a}) GB_2G^{-1} - GB_2G^{-1} M(z_{\mathsf a}) \| \ge 1.9663 \times 10^{-4}
    \end{equation}
    However, from equation \eqref{eq:GBlimGinv} we also obtain
    \begin{equation}
    \begin{aligned}
        \| M(z_{\mathsf a}) GB_2G^{-1}  - M(z_{\mathsf a}) GBG^{-1} \| &\le \| M(z_{\mathsf a}) \| \|GB_2G^{-1}  -  GBG^{-1} \| \le 10^{-5} \\
        \|  GB_2G^{-1}M(z_{\mathsf a})  - GBG^{-1}M(z_{\mathsf a}) \| &\le  \|GB_2G^{-1}  -  GBG^{-1} \| \| M(z_{\mathsf a}) \| \le 10^{-5} \\
    \end{aligned}
    \end{equation}
    This implies $\|M(z_{\mathsf a}) GBG^{-1} - GBG^{-1}M(z_{\mathsf a})\|$ cannot be zero. This contradiction proves the lemma.

\end{proof}

\subsection{Full and partial fiber returns for the periodic word}

We now show that the matrix cocycle $M^{(j)}(z_n)$ in the beginning part of $\mathsf a^n$-part of the $z_n$ can be approximated by $GR_{j\omega_{\mathsf a}}G^{-1}$, and the matrix $P_n$ for the full $\mathsf a^n\mathsf b$ period converges to $GR_{n\omega_{\mathsf a}}BG^{-1}$.

\begin{corollary}\label{cor:period-estimates}
For every $n\ge1$,
\begin{equation}\label{eq:period-asymptotic}
\|G^{-1}P_nG-R_{n\omega_{\mathsf a}}B\|
\le4\cdot400^{-n}.
\end{equation}
For $0\le j\le n$,
\begin{equation}\label{eq:partial-rotation}
\|M^{(j)}(z_n)-GR_{j\omega_{\mathsf a}}G^{-1}\|
\le2\cdot600^{-(n-j)}.
\end{equation}
Moreover, we have the uniform bound on norms:
\begin{equation}\label{eq:partial-uniform-bound}
\|P_n\|+\|P_n^{-1}\|
+\|M^{(j)}(z_n)\|+\|M^{(j)}(z_n)^{-1}\|
\le10.
\end{equation}
\end{corollary}

\begin{proof}
Fix $n\ge1$. The words $z_n$ and $\sigma^{-n}z_\delta$ agree in their first
$2n+1$ symbols. 

For Equation~\eqref{eq:partial-rotation}, 
we apply~\eqref{eq:uniform-fiber-bound} with $\ell=j$ and $N=n-j$, which
gives
\[
\bigl\|M^{(j)}(z_{\mathsf a})^{-1}M^{(j)}(z_n)-I_2\bigr\|
\le600^{-(n-j)}.
\]
By Lemma~\ref{lem:constant-branch},
$$
M^{(j)}(z_{\mathsf a})=GR_{j\omega_{\mathsf a}}G^{-1},
\qquad
\|G^{-1}\|=\|G\|<1.03. $$
Consequently,
\[
\|M^{(j)}(z_n)-GR_{j\omega_{\mathsf a}}G^{-1}\|
\le\|G\|^2\,600^{-(n-j)}
\le2\cdot600^{-(n-j)},
\]
proving~\eqref{eq:partial-rotation}.

The same relative estimate gives
\[
\|M^{(j)}(z_n)\|
\le\|G\|^2\bigl(1+600^{-(n-j)}\bigr)
\le2\|G\|^2.
\]
Since $P_n=M^{(n)}(z_n)M(\sigma^nz_n)$,
Corollary~\ref{lem:singMsu} yields
\[
\|P_n\|\le2\cdot1.04\|G\|^2.
\]
A matrix in $\SL_2(\R)$ and its inverse have the same Euclidean
operator norm. Hence
\[
\begin{aligned}
&\|P_n\|+\|P_n^{-1}\|
+\|M^{(j)}(z_n)\|+\|M^{(j)}(z_n)^{-1}\|\\
&\qquad\le4(1+1.04)\|G\|^2<10,
\end{aligned}
\]
which proves~\eqref{eq:partial-uniform-bound}.

 Now we prove Equation~\eqref{eq:period-asymptotic}. Applying~\eqref{eq:uniform-fiber-bound}, with
$\ell=n+1$ and $N=n$, gives
\[
\bigl\|P_n^{-1}M^{(n+1)}(\sigma^{-n}z_\delta)-I_2\bigr\|
\le600^{-n}.
\]
By~\eqref{eq:GBrGinv} and~\eqref{eq:GBlimGinv}, with $r=n$,
\[
G^{-1}M^{(n+1)}(\sigma^{-n}z_\delta)G
=R_{n\omega_{\mathsf a}}B_n,
\qquad
\|G(B_n-B)G^{-1}\|\le1.3\,400^{-n}.
\]
Together with the bound $\|P_n\|\le2\cdot1.04\|G\|^2$
proved above, these give
\[
\begin{aligned}
\|G^{-1}P_nG-R_{n\omega_{\mathsf a}}B\|
&\le
\|G\|^2\|P_n\|\,600^{-n}
+\|G\|^2\|G(B_n-B)G^{-1}\|\\
&\le
2\cdot1.04\|G\|^4\,600^{-n}
+1.3\|G\|^2\,400^{-n}
<4\cdot400^{-n},
\end{aligned}
\]
proving~\eqref{eq:period-asymptotic}.
\end{proof}

In particular, whenever $j\to\infty$ and $n-j\to\infty$,
\begin{equation}\label{eq:bulk-limit}
\sigma^jz_n\longrightarrow z_{\mathsf a}\quad\text{in }\Omega,\qquad
G^{-1}M^{(j)}(z_n)G-R_{j\omega_{\mathsf a}}\longrightarrow0.
\end{equation}
Continuity of the return time $h$ then gives $h(\sigma^jz_n)\to h(z_{\mathsf a})$.

\subsection{When is $P_n$ elliptic?}

When $n_j\omega_{\mathsf a}\to \theta\in \T^1$ as $n_j\to \infty$, the limiting matrix at phase $\theta$ is conjugate to $R_\theta B$. We now identify a set $I_{\mathrm{ell}}\subset \T^1$, such that $R_\theta B$ is elliptic if and only if $\theta\in I_{\mathrm{ell}}$.

Writing $B=(b_{ij})$, one has
\begin{equation*}
\tr(R_\theta B)=(b_{11}+b_{22})\cos\theta+(b_{12}-b_{21})\sin\theta = K\sin(\theta+\psi).
\end{equation*}
Moreover,
\begin{equation*}
K^2 = (b_{11}+b_{22})^2+(b_{12}-b_{21})^2
=\sum b_{ij}^2+2\det B>4,
\end{equation*}
because $\det B=1$ and $B\notin\SO(2)$.
Thus the elliptic phase set
\begin{equation*}
I_{\mathrm{ell}}=\{\theta\in\T^1:|\tr(R_\theta B)|<2\}
\end{equation*}
is nonempty and open. In fact, it is the union of two antipodal open intervals.
At each component endpoint $\theta_*$, the matrix $R_{\theta_*}B$ is nontrivial parabolic (i.e. $R_{\theta_*}B\neq \pm I_2$).


\begin{proposition}\label{prop:admissible}
For every compact interval $J\Subset I_{\mathrm{ell}}$, there exist $N_J,C_J>0$ such that for any $n\ge N_J$ with $n\omega_{\mathsf a}\in J$ the fiber shear $P_n$ is elliptic; and for any $k\in \Z$, $
\|P_n^k\|\le C_{J}.$

Whenever $P_n$ is elliptic, its rotation angle is irrational modulo $\pi$, and for every $E\in\SL_2(\R)$ the orbit of $C(z_n,E)\Z^3$ has compact closure. 
\end{proposition}

\begin{proof}
For the first assertion, we note that since $J\Subset I_{\mathrm{ell}}$, there exists $\epsilon>0$ such that $|\tr(R_\theta B)|<2-\epsilon$ for any $\theta\in J$. Therefore by Corollary \ref{cor:period-estimates}, there exists $N_J>0$ such that $|\tr(P_n)|<2$ if $n\ge N_J$ and $n\omega_{\mathsf a}\in J$. 
Write $\tr(G^{-1}P_nG)=2\cos\beta_n$; then $|\sin\beta_n|$ is uniformly bounded below.
For $k\ge1$, Cayley--Hamilton gives
\[(G^{-1}P_nG)^k=\frac{\sin(k\beta_n)}{\sin\beta_n}G^{-1}P_nG-\frac{\sin((k-1)\beta_n)}{\sin\beta_n}I_2.\] This bounds positive powers uniformly. The same argument applies to negative powers, and conjugation by the fixed matrix $G$ proves the claim.

For the second, the marked return through one period gives
\[
\begin{split}
&C(z_n,I_2)\Gamma_{\mathsf a}^n\Gamma_{\mathsf b}C(z_n,I_2)^{-1}\\
&\hspace{1cm}=\operatorname{diag}\bigl(e^{2\tau_n},e^{-\tau_n}P_n\bigr).
\end{split}
\]
Since $\tau_n>0$, the integral matrix $\Gamma_{\mathsf a}^n\Gamma_{\mathsf b} \in \SL_3(\Z)$ has one eigenvalue ($e^{2\tau_n}$) greater than one and a conjugate pair ($e^{-\tau_n\pm i\theta_n}$) of modulus less than one.
Its monic cubic characteristic polynomial $\chi(\Gamma_{\mathsf a}^n\Gamma_{\mathsf b}) \in \Z[x]$ is irreducible: the constant term is $-1$ and a rational root would be $1$ or $-1$, and neither occurs by the modulus argument. With a conjugate pair of roots, $\chi$ has the discriminant $\Delta(\chi) < 0$ and the Galois group is $\operatorname{Gal}(\chi) = S_3$. 
The same field automorphism argument as in Lemma~\ref{lem:constant-branch} therefore proves irrationality of $\theta_n/\pi$ where $\theta_n$ is the angle of $P_n$.
The powers of $P_n$ are bounded, so~\eqref{eq:complete-orbit} represents the orbit by a bounded set of matrices with bounded inverses.
This gives compactness of the orbit closure.
\end{proof}

We call $n$ \emph{admissible} when $P_n$ is elliptic.



\section{Orbit closures and limiting sets}\label{sec:depths}

Each elliptic periodic word $(\mathsf a^n\mathsf b)^\infty$ gives an orbit closure that is a two-torus.
As the $\mathsf a$-block grows and its rotation phase converges to $\theta$, positions near its beginning, near its end, and deep in the middle give three limiting families.
Together these form a continuous compact family $\mathcal K(\theta)$ whose depth tends to zero at a parabolic endpoint.

\subsection{Exact periodic-code orbit closures}

For admissible $n$, let
\begin{equation*}
\mathcal E_n=\cl{\{P_n^k:k\ge0\}}.
\end{equation*}

By~\eqref{eq:complete-orbit}, the corresponding orbit closure is
\begin{equation}\label{eq:On-explicit}
\begin{split}
\mathcal O_n
&=\cl{\{g_tC(z_n,I_2)\Z^3:t\ge0\}}\\
&=\left\{g_rC(\sigma^jz_n,EM^{(j)}(z_n))\Z^3:
\begin{array}{l}
E\in\mathcal E_n,\ 0\le j\le n,\\
0\le r\le h(\sigma^jz_n)
\end{array}\right\}.
\end{split}
\end{equation}
Positive powers of $P_n$ visit every open subset of $\mathcal E_n$ infinitely often.
Consequently every tail of the orbit has closure $\mathcal O_n$, which is compact and invariant under the full flow.

\begin{theorem}\label{thm:invariant-two-tori}
For every admissible $n$, the set $\mathcal O_n$ is a compact embedded two-torus in $X_3$, invariant under $g_t$.
Every $g_t$-orbit on each of these tori is dense in that torus.
\end{theorem}

\begin{proof}
For an admissible $n$, put $\Lambda_n=C(z_n,I_2)\Z^3$ and consider the abelian subgroup
\[
\mathcal A_n=\{g_t\operatorname{diag}(1,E):t\in\R,\ E\in\mathcal E_n\}
\cong\R\times S^1.
\]
The period identity is
\begin{equation}\label{eq:torus-period-identification}
g_{\tau_n}\Lambda_n=\operatorname{diag}(1,P_n)\Lambda_n.
\end{equation}
The stabilizer $S$ of $\Lambda_n$ in $\mathcal A_n$ contains the element with coordinates $(\tau_n,P_n^{-1})$.
It is discrete, since the stabilizer of any lattice in $\SL_3(\R)$ is discrete.
The quotient $\mathcal A_n/S$ is compact: every coset has a representative with $0\le t\le \tau_n$ and $E\in\mathcal E_n$, so $\mathcal A_n/S$ is a connected compact abelian Lie group of dimension two, hence a two-torus.
Therefore the induced map from $\mathcal A_n/S$ to $X_3$ is an embedding of a two-torus.

The positive powers of $P_n$ are dense in $\mathcal E_n$.
Together with~\eqref{eq:torus-period-identification}, this shows that the forward $g_t$-orbit of $\Lambda_n$ is dense in $\mathcal A_n\Lambda_n$.
Hence this embedded torus is precisely $\mathcal O_n$. Since $\mathcal A_n$ is abelian, translation by any of its elements carries that dense orbit to the orbit of the translated point; every orbit on the torus is therefore dense.

\end{proof}

\subsection{Derivation of the limiting models}

We seek to describe the limit of $\mathcal O_n$ as $n\to\infty$ and $n\omega_{\mathsf a}\to\theta$ within an elliptic interval.

Fix a connected component $I_0$ of $I_{\mathrm{ell}}$, and an endpoint $\theta_*\in\partial I_0$.

On $I_0$, choose continuously $H_\theta\in\SL_2(\R)$ and $\beta(\theta)\in(0,2\pi)$ such that
\begin{equation*}
H_\theta(R_{\theta}B)H_\theta^{-1}=R_{\beta(\theta)}.
\end{equation*}

Such a choice exists locally throughout the elliptic locus: the unique fixed point of the associated M\"obius transformation in the upper half-plane $\mathbb H$ varies continuously, and a matrix sending the fixed point to $i\in \mathbb H$ gives $H_\theta$.

The period asymptotic~\eqref{eq:period-asymptotic} identifies the limiting elliptic subgroup in the original fiber coordinates as
\begin{equation*}
\mathcal E(\theta)=\{GH_\theta^{-1}R_\xi H_{\theta}G^{-1}:\xi\in\T^1\}.
\end{equation*}
This is the full compact subgroup containing $G(R_{\theta}B)G^{-1}$, independently of the choice of $H_\theta$.
The full circle appears because each admissible $P_n$ has irrational angle, even when the limiting angle $\beta(\theta)$ is rational modulo $\pi$.

\paragraph{Beginning of the block.}
At the $0$-place of the word, the nearest $\mathsf b$ is at position $-1$, while the next $\mathsf b$ moves to infinity as $n\to \infty$.
Thus $z_n\to\sigma z_\delta$ in $\Omega$.
More generally, at any fixed position $r\ge0$,
\begin{equation}\label{eq:beginning-limit}
\sigma^rz_n\longrightarrow\sigma^{r+1}z_\delta,\qquad
M^{(r)}(z_n)\longrightarrow M^{(r)}(\sigma z_\delta)=G R_{r\omega_{\mathsf a}}G^{-1}.
\end{equation}
The first limit follows by checking each fixed window of symbols, and the second follows from continuity of the finite product $M^{(r)}$.
The equality on the right uses the fact that the future of $\sigma z_\delta$ consists entirely of $\mathsf a$'s.
The limiting fiber at this position is therefore of the form $E M^{(r)}(\sigma z_\delta)$, with $E\in\mathcal E(\theta)$.

\paragraph{End of the block.}
At position $n-r$, with $r\ge0$ fixed, the next $\mathsf b$ is at position $r$ and every other $\mathsf b$ moves out of any fixed window.
Hence
$\sigma^{n-r}z_n\to\sigma^{-r}z_\delta$.
The remaining suffix is $\mathsf a^r\mathsf b$, so continuity of its fixed-length product gives
\begin{equation}\label{eq:end-limit}
\begin{aligned}
\sigma^{n-r}z_n&\longrightarrow\sigma^{-r}z_\delta,\\
M^{(r+1)}(\sigma^{n-r}z_n)&\longrightarrow M^{(r+1)}(\sigma^{-r}z_\delta).
\end{aligned}
\end{equation}
Since $P_n=M^{(n-r)}(z_n)M^{(r+1)}(\sigma^{n-r}z_n)$ and $\mathcal E_nP_n=\mathcal E_n$,
\begin{equation}\label{eq:end-circle}
\mathcal E_nM^{(n-r)}(z_n)
=\mathcal E_nM^{(r+1)}(\sigma^{n-r}z_n)^{-1}.
\end{equation}
Thus the end family has fiber coordinate $E M^{(r+1)}(\sigma^{-r}z_\delta)^{-1}$.

\paragraph{Bulk of the block.}
If both $j_i$ and $n_i-j_i$ tend to infinity, every fixed window around position $j_i$ eventually contains only $\mathsf a$'s.
Thus $\sigma^{j_i}z_{n_i}\to z_{\mathsf a}$.
Moreover,~\eqref{eq:bulk-limit} gives, along a subsequence with $j_i\omega_{\mathsf a}\to\xi$,
\begin{equation*}
\sigma^{j_i}z_{n_i}\longrightarrow z_{\mathsf a},\qquad
M^{(j_i)}(z_{n_i})\longrightarrow GR_\xi G^{-1}.
\end{equation*}
The phase $\theta$ records the length of the complete period, while $\xi$ records the position inside its long $\mathsf a$-block.
Irrationality of $\omega_{\mathsf a}/\pi$ allows every $\xi\in\T^1$ with both distances tending to infinity, as verified in the proof below.
Thus the bulk family includes every fiber coordinate $EGR_\xi G^{-1}$, with $E\in\mathcal E(\theta)$ and $\xi\in\T^1$.

Continuity of $h$ gives the limits of the complete time intervals between returns.
For $\theta\in I_0$, define $$\mathcal K(\theta)=\overline{\mathcal K^{begin}(\theta)\cup \mathcal K^{end}(\theta)\cup\mathcal K^{bulk}(\theta)}$$ as the closure of the union of the beginning, end, and bulk families:
\begin{align}
\mathcal K^{begin}(\theta)=&\left\{
 g_tC\bigl(\sigma^{r+1}z_\delta,E M^{(r)}(\sigma z_\delta)\bigr)\Z^3:
 \begin{array}{l}
 E\in\mathcal E(\theta),\ r\in\N_0,\\
 0\le t\le h(\sigma^{r+1}z_\delta)
 \end{array}
\right\},
\label{eq:limit-beginning}\\
\mathcal K^{end}(\theta)=&\left\{
 g_tC\bigl(\sigma^{-r}z_\delta,E M^{(r+1)}(\sigma^{-r}z_\delta)^{-1}\bigr)\Z^3:
 \begin{array}{l}
 E\in\mathcal E(\theta),\ r\in\N_0,\\
 0\le t\le h(\sigma^{-r}z_\delta)
 \end{array}
\right\},
\\
\mathcal K^{bulk}(\theta)=&\left\{
 g_tC\bigl(z_{\mathsf a},EGR_\xi G^{-1}\bigr)\Z^3:
 E\in\mathcal E(\theta),\ \xi\in\T^1,\ 0\le t\le h(z_{\mathsf a})
\right\}.
\label{eq:limit-bulk}
\end{align}

The following proposition shows that these positional limits determine the Hausdorff limit of the orbit closures.

\begin{proposition}
    \label{prop:Hausdorff-limit}
For every $\theta\in I_0$ and every sequence of positive integers $(n_i)_{i\ge1}$ satisfying
\[
n_i\to\infty,
\qquad
n_i\omega_{\mathsf a}\to \theta\text{ in }\T^1,
\qquad
n_i\text{ admissible},
\]
one has
\begin{equation*}
d_H(\mathcal O_{n_i},\mathcal K(\theta))\to0.
\end{equation*}
The sets $\mathcal K(\theta)$ are compact and flow-invariant, and the map $\theta\mapsto\mathcal K(\theta)$ is Hausdorff continuous on $I_0$.
\end{proposition}

\begin{proof}
Fix an admissible sequence $n_i\to\infty$ with $n_i\omega_{\mathsf a}\to \theta\in I_0$.
By~\eqref{eq:period-asymptotic}, $G^{-1}P_{n_i}G\to R_{\theta}B$.
Choose elliptic conjugators $H_i\to H_\theta$ with
$H_iG^{-1}P_{n_i}GH_i^{-1}=R_{\beta_i}$.
Admissibility gives
\[
\mathcal E_{n_i}=\{GH_i^{-1}R_\xi H_iG^{-1}:\xi\in\T^1\},
\]
so the parametrizations converge uniformly and
\begin{equation*}
d_H(\mathcal E_{n_i},\mathcal E(\theta))\to0.
\end{equation*}
Together with~\eqref{eq:partial-uniform-bound} and the return-time bounds, this places all $\mathcal O_{n_i}$ in one compact subset of $X_3$.

For any sequence $x_i\in\mathcal O_{n_i}$, choose representations in~\eqref{eq:On-explicit} with positions $j_i$.
After passing to a subsequence, one of three cases holds:
$j_i=r$ is fixed; $n_i-j_i=r$ is fixed; or both $j_i,n_i-j_i$ tend to infinity.
The beginning limit~\eqref{eq:beginning-limit}, the end identities~\eqref{eq:end-limit}--\eqref{eq:end-circle}, and the bulk limit~\eqref{eq:bulk-limit}, respectively, place every accumulation point in $\mathcal K(\theta)$.
This proves the upper Hausdorff inclusion.

Conversely, fixed positions approximate each point of the boundary families.
For a bulk point with phase $\xi$ and matrix $E\in\mathcal E(\theta)$, choose $E_i\in\mathcal E_{n_i}$ with $E_i\to E$.
Irrational rotation has uniformly bounded gaps between visits to any open arc: density of its forward orbit and compactness of $\T^1$ give a finite cover by inverse translates of that arc.
Consequently one can choose $j_i$ with
\[
j_i\to\infty,\qquad n_i-j_i\to\infty,\qquad j_i\omega_{\mathsf a}\to\xi.
\]
Equation~\eqref{eq:bulk-limit} then gives $\sigma^{j_i}z_{n_i}\to z_{\mathsf a}$ and $E_iM^{(j_i)}(z_{n_i})\to EGR_\xi G^{-1}$, together with convergence of the return times.
Thus all three families, and their closure, are approximated by $\mathcal O_{n_i}$.

The beginning products are powers of $M(z_{\mathsf a})$, so they are uniformly bounded.
The suffix products and their inverses are uniformly bounded by the convergence rate in~\eqref{eq:defect-matrix}.
Hence all three defining families have uniformly bounded representatives and inverses when $\theta$ ranges over a compact subinterval of $I_0$; this proves compactness of $\mathcal K(\theta)$.
Its flow invariance follows from the Hausdorff convergence of the invariant sets $\mathcal O_{n_i}$.
Finally, the only dependence on $\theta$ in~\eqref{eq:limit-beginning}--\eqref{eq:limit-bulk} is through the continuously parametrized circle $\mathcal E(\theta)$.
The remaining factors are uniformly bounded, which proves Hausdorff continuity.
\end{proof}



\subsection{The invariant three-dimensional bulk}

The bulk family $\mathcal K^{bulk}(\theta)$ has two circle parameters, from the complete period and the position inside the $\mathsf a$-block.
We verify that these directions, together with flow time, give an immersion of rank three. 

Write\begin{equation*}\mathcal E_{\mathsf a}=\{GR_\xi G^{-1}:\xi\in\T^1\},\quad\Lambda_{\mathsf a}=C(z_{\mathsf a},I_2)\Z^3,\quad \mathcal O_{\mathsf a}=\overline{\{g_t\Lambda_{\mathsf a}:t\ge0\}}.\end{equation*}  The constant-branch return and irrationality of $\omega_{\mathsf a}/\pi$ give\begin{equation}\label{eq:constant-torus}\mathcal O_{\mathsf a}=\{g_tC(z_{\mathsf a},K)\Z^3:K\in\mathcal E_{\mathsf a},\ 0\le t\le h(z_{\mathsf a})\}.\end{equation}The same proof as Theorem \ref{thm:invariant-two-tori} shows that $\mathcal O_{\mathsf a}$ is also an embedded two-torus.

By Equations \eqref{eq:limit-bulk} and ~\eqref{eq:constant-torus}, $\mathcal K^{bulk}(\theta)$ can also be described by
\begin{equation}\label{eq:bulk-as-tori}
\mathcal K^{bulk}(\theta)
=\bigcup_{E\in\mathcal E(\theta)}\operatorname{diag}(1,E)\mathcal O_{\mathsf a}
\subset\mathcal K(\theta).
\end{equation}

\begin{theorem}\label{thm:invariant-three-tori}
For every $\theta\in I_0$, the bulk $\mathcal K^{bulk}(\theta)$ is the image of the smooth immersion
\begin{equation}\label{eq:bulk-immersion}
\mathcal E(\theta)\times\mathcal O_{\mathsf a}\longrightarrow X_3:\quad
(E,\Lambda)\longmapsto\operatorname{diag}(1,E)\Lambda.
\end{equation}
Its domain is a compact three-torus, and its image is invariant under $g_t$.
For each fixed $E$, the image of $\{E\}\times\mathcal O_{\mathsf a}$ is an embedded invariant two-torus on which every orbit is dense.
\end{theorem}

\begin{proof}
Both $\mathcal E(\theta)$ and $\mathcal E_{\mathsf a}$ are circle subgroups of $\SL_2(\R)$, and they are distinct.
Indeed, equality would imply that their element $G(R_\theta B)G^{-1}$ belongs to $G\SO(2)G^{-1}$, forcing $B\in\SO(2)$, contrary to Lemma~\ref{lem:nonorthogonal-defect}.
Their one-dimensional Lie algebras are therefore distinct.
Choose nonzero generators $X_\theta$ and $X_{\mathsf a}$ of these Lie algebras.

Locally, write a point of $\mathcal O_{\mathsf a}$ as $g_t\operatorname{diag}(1,K)\Lambda_{\mathsf a}$, with $K\in\mathcal E_{\mathsf a}$.
A local lift of~\eqref{eq:bulk-immersion} to $\SL_3(\R)$ is
\[
(E,t,K)\longmapsto
g_t\operatorname{diag}(1,EK)C(z_{\mathsf a},I_2).
\]
After trivializing tangent vectors by invertible left and right translations, its two circle directions are represented in the lower block by
$K^{-1}X_\theta K$ and $X_{\mathsf a}$.
They are linearly independent: dependence would imply
$X_\theta\in\R X_{\mathsf a}$, since $K$ centralizes $X_{\mathsf a}$.
The time direction has a nonzero upper-left entry, whereas both circle directions have zero upper-left entry.
Thus the differential has rank three everywhere.
Projection to $X_3$ preserves this rank because $\SL_3(\Z)$ is discrete.

The domain is the product of a circle and the two-torus $\mathcal O_{\mathsf a}$.
Equation~\eqref{eq:bulk-as-tori} identifies the image, and commutation of $\operatorname{diag}(1,E)$ with $g_t$ proves invariance.
For each fixed $E$, left translation carries the embedded minimal two-torus $\mathcal O_{\mathsf a}$ to the stated embedded invariant two-torus.
\end{proof}

\subsection{Continuous maxima and parabolic degeneration}\label{sec:proper-maxima}

From now on, fix a continuous proper function $F:X_3\to[0,\infty)$ as in Theorem \ref{thm:proper-ray}.

For $E\in\SL_2(\R)$, $n$ admissible, and $\theta\in I_0$, write
\[
\begin{aligned}
\Delta_n(E)&=\max_{\Lambda\in\mathcal O_n}F\bigl(\operatorname{diag}(1,E)\Lambda\bigr),\\
\Delta(E,\theta)&=\max_{\Lambda\in\mathcal K(\theta)}F\bigl(\operatorname{diag}(1,E)\Lambda\bigr).
\end{aligned}
\]
Since $\operatorname{diag}(1,E)$ commutes with $g_t$, Theorem~\ref{thm:invariant-two-tori} gives
\[
\Delta_n(E)=\limsup_{t\to\infty}F\bigl(g_tC(z_n,E)\Z^3\bigr).
\]

We claim that $\Delta(E,\theta)$ is continuous with respect to the parameters, and converges to infinity as $\theta$ approaches an endpoint of $I_{\mathrm{ell}}$.

\begin{proposition}\label{prop:proper-maxima}
The function $(E,\theta)\mapsto\Delta(E,\theta)$ is continuous on $\SL_2(\R)\times I_0$.
If $n_i\to\infty$, $n_i\omega_{\mathsf a}\to\theta\in I_0$, and $n_i$ is admissible, then
\begin{equation}\label{eq:translated-depth-convergence}
\Delta_{n_i}(E)\longrightarrow\Delta(E,\theta),
\end{equation}
uniformly for $E$ in compact subsets of $\SL_2(\R)$.
Moreover, for $\theta\to\theta_*,\ \theta\in I_0$ we have
\begin{equation}\label{eq:proper-degeneration}
\Delta(E,\theta)\longrightarrow\infty
\end{equation}
uniformly for $E$ in compact subsets of $\SL_2(\R)$.
\end{proposition}

\begin{proof}
Proposition~\ref{prop:Hausdorff-limit} gives Hausdorff continuity of $\mathcal K(\theta)$ and convergence of $\mathcal O_{n_i}$ to $\mathcal K(\theta)$.
For $E$ in a fixed compact set and $\theta$ in a compact subinterval of $I_0$, the translated sets lie in a common compact subset of $X_3$.
Uniform continuity of $F$ on the compact subset proves joint continuity and~\eqref{eq:translated-depth-convergence}.

As $\theta\to\theta_*$, one has $\|H_\theta\|\to\infty$: otherwise a convergent subsequence would conjugate the nontrivial parabolic matrix $R_{\theta_*}B$ to a rotation.
The least singular value of $H_\theta^{-1}R_{\pi/2}H_\theta$ is $\|H_\theta\|^{-2}$.

Write $\bar u=u(z_{\mathsf a})$.
Choose $\xi_\theta\in\T^1$ so that $R_{\xi_\theta}G^{-1}\bar u$ lies in the contracting singular direction, and put
\[
\Lambda_\theta=C\bigl(z_{\mathsf a},G H_\theta^{-1}R_{\pi/2}H_\theta R_{\xi_\theta}G^{-1}\bigr)\Z^3
\in\mathcal K(\theta).
\]
The image of $(1,0,0)^T$ under the chart gives
\[
\sys(g_t\Lambda_\theta)\le C\max\{e^{-2t},e^t\|H_\theta\|^{-2}\}
\qquad(t\ge0)
\] for a uniform constant $C$ that depends only on the norm and the compact set.
Taking $t_\theta=\frac23\log\|H_\theta\|$ and using flow invariance, we obtain
\[
g_{t_\theta}\Lambda_\theta\in\mathcal K(\theta),\qquad
\sys(g_{t_\theta}\Lambda_\theta)\le C\|H_\theta\|^{-4/3}\longrightarrow0.
\]

The same systole bound holds after left translation by $\operatorname{diag}(1,E)$, uniformly for $E$ in any fixed compact subset of $\SL_2(\R)$.
Every nonempty sublevel set of $F$ is compact and therefore has a positive lower bound for its systole.
Consequently,
\[
\Delta(E,\theta)\ge F\bigl(\operatorname{diag}(1,E)g_{t_\theta}\Lambda_\theta\bigr)\longrightarrow\infty
\]
uniformly on such compact sets, proving~\eqref{eq:proper-degeneration}.
\end{proof}

\section{Realizing a ray by concatenation}\label{sec:realization}

We now prove Theorem~\ref{thm:proper-ray} by concatenating repeated periodic blocks.
At each step, the next periodic maximum is adjusted to account for the current fiber matrix.
Its repetition length is chosen both to visit that maximum and to make the change in the terminal fiber summably small.
The uniform finite-word estimate then compares every return segment of the resulting orbit with the corresponding periodic segment.

\begin{proof}[Proof of Theorem~\ref{thm:proper-ray}]
Fix $\theta_0\in I_0$, put $T=\max\{1,\Delta(I_2,\theta_0)\}$, and let $A>T$.
We shall construct $z\in\Omega$ with $z(i)=\mathsf a$ for $i<0$ and
\[
z[0,\infty)=(\mathsf a^{n_1}\mathsf b)^{L_1}(\mathsf a^{n_2}\mathsf b)^{L_2}\cdots,
\qquad n_1<n_2<\cdots,
\]
such that
\begin{equation}\label{eq:prescribed-limsup}
\limsup_{t\to\infty}F\bigl(g_tC(z,I_2)\Z^3\bigr)=A.
\end{equation}

\paragraph{Step 1: Allowing a small shift in the fiber.}
By Proposition~\ref{prop:proper-maxima}, choose $\theta_+\in I_0$ with
\[
\Delta(I_2,\theta_0)<A<\Delta(I_2,\theta_+).
\]
Let $J\Subset I_0$ be the closed interval between $\theta_0$ and $\theta_+$.
For $r\ge0$, write
\[
B_r(I_2)=\{E\in\SL_2(\R):\|E-I_2\|\le r\}.
\]
By continuity, choose $0<\delta<1$ such that
\[
\Delta(E,\theta_0)<A<\Delta(E,\theta_+),\quad
\forall E\in B_\delta(I_2).
\]
For each such $E$, fix $\theta(E,A)\in\operatorname{int}J$ satisfying
\[
\Delta(E,\theta(E,A))=A.
\]
Approximating this phase by $n\omega_{\mathsf a}$ and using Proposition~\ref{prop:admissible} and~\eqref{eq:translated-depth-convergence}, we obtain: for every $E\in B_\delta(I_2)$, $\varepsilon>0$, and $N\ge1$, there is an admissible $n\ge N$ such that
\begin{equation}\label{eq:adjusted-depth-choice}
n\omega_{\mathsf a}\in J,\qquad |\Delta_n(E)-A|<\varepsilon.
\end{equation}

\paragraph{Step 2: A summable induction choice.}
Fix $e_i=2^{-(i+1)}$ for all $i\ge1$. We first choose some uniform constants. By Proposition \ref{prop:admissible} and~\eqref{eq:partial-uniform-bound}, we may choose $N_J\in\N$ and $K\ge1$ such that for every admissible $n\ge N_J$ with $n\omega_{\mathsf a}\in J$ and every $j\in\N_0$,
\[
\|M^{(j)}(z_n)\|+\|(M^{(j)}(z_n))^{-1}\|\le K.
\]Fix the separate constant
$
D=6K+3$.

For every $i\ge1$, fix $\rho_i\in(0,1)$ such that, for all $w,w'\in\Omega$ and $Q,Q'\in\SL_2(\R)$ with $\|Q\|,\|Q'\|\le 2K+1$, the condition
\[
\|u(w)-u(w^\prime)\|_\infty+\|v(w)-v(w^\prime)\|_\infty+\|Q-Q'\|<\rho_i
\]
implies $ |\max_{0\le t\le h(w)}F\bigl(g_tC(w,Q)\Z^3\bigr)-\max_{0\le t\le h(w^\prime)}F\bigl(g_tC(w^\prime,Q^\prime)\Z^3\bigr)|< e_i$.

We now choose $n_i,L_i\in\N$ and matrices
\[
E_i\in B_{e_i\delta}(I_2),\qquad F_i\in B_{(1-2e_i)\delta}(I_2)\qquad(i\ge1),
\]
where $E_i$ is the fiber multiplier of the $i$th periodic block and $F_i$ is the cumulative fiber matrix after the $i$th choice. We denote by $r_i$ the word length $r_i=\sum_{j=0}^i L_j(n_j+1)$

Set $n_0=L_0=r_0=0$ and $E_0=F_0=I_2$.

At step $i\ge1$, suppose that $n_j,L_j,r_j, E_j,F_j$ have been constructed for $0\le j<i$, with
\[
E_j\in B_{e_j\delta}(I_2),\qquad F_j\in B_{(1-2e_j)\delta}(I_2)\qquad(1\le j<i).
\]

Take $\theta_i=\theta(F_{i-1},A)$.
By~\eqref{eq:adjusted-depth-choice}, choose an admissible integer $n_i>\max\{n_{i-1},N_J\}$ such that
\[
n_i\omega_{\mathsf a}\in J,\qquad |\Delta_{n_i}(F_{i-1})-A|<e_i,\qquad 600^{-n_i}<\frac{\min\{\delta e_{i+1},\rho_{i+1}\}}{D}.
\]
For this $n_i$, Theorem~\ref{thm:invariant-two-tori} and Proposition~\ref{prop:admissible} give, with $L\in\N$,
\[
\lim_{L\to\infty}\max_{0\le t\le L\tau_{n_i}}F\bigl(g_tC(z_{n_i},F_{i-1})\Z^3\bigr)=\Delta_{n_i}(F_{i-1}),\qquad \liminf_{L\to\infty}\|P_{n_i}^{L}-I_2\|=0.
\]
Choose an integer $L_i\ge i$ such that
\begin{equation}\label{eq:chosen-repetition-length}
|\max_{0\le t\le L_i\tau_{n_i}}F\bigl(g_tC(z_{n_i},F_{i-1})\Z^3\bigr)-\Delta_{n_i}(F_{i-1})|<e_i,\qquad D\|P_{n_i}^{L_i}-I_2\|<\delta e_i.
\end{equation}

Take
$
E_i=P_{n_i}^{L_i},$ and $ r_i=r_{i-1}+L_i(n_i+1).$ Then $E_i\in B_{e_i\delta}(I_2)$.

Let $z^{(i)}\in\Omega$ be the bi-infinite word with all negative symbols equal to $\mathsf a$ and future
\[
z^{(i)}[0,\infty)=(\mathsf a^{n_1}\mathsf b)^{L_1}\cdots(\mathsf a^{n_i}\mathsf b)^{L_i}(\mathsf a^{n_i}\mathsf b)^\infty,
\]
and define $F_i=M^{(r_i)}(z^{(i)})$.
We note that $F_1=E_1$ and for $i\ge2$, since $z^{(i-1)}$ agrees with $z^{(i)}$ up to the first $r_{i-1}+n_{i-1}$ symbols, by Lemma~\ref{lem:fiber-bound},
\[
\|F_i-F_{i-1}E_i\|\le \|F_{i-1}\|\|E_i\|600^{-n_{i-1}}\le4\,600^{-n_{i-1}}.
\]
Summation yields
\[
\begin{aligned}
\|F_i-I_2\|&\le D\sum_{j=1}^{i-1}600^{-n_j}+D\sum_{j=1}^i\|E_j-I_2\|\\
&<\delta\sum_{j=2}^ie_j+\delta\sum_{j=1}^ie_j<2\delta\sum_{j=1}^ie_j=(1-2e_i)\delta.
\end{aligned}
\]
This shows that $F_i\in B_{(1-2e_i)\delta}(I_2)$.

\paragraph{Step 3: Comparing the entire word.}
Let $s_i=T_{r_i}(z)$ for $i\ge0$ and $s_0=0$.
Fix $i\ge2$.
Since $z$ and $z^{(i-1)}$ agree through the first $r_{i-1}+n_{i-1}$ symbols, Lemma~\ref{lem:fiber-bound} gives
\[
\|M^{(r_{i-1})}(z)-F_{i-1}\|\le2\cdot \,600^{-n_{i-1}}.
\]
For each integer $j$ with $0\le j<L_i(n_i+1)$, put
\[
w=\sigma^{r_{i-1}+j}z,\quad w'=\sigma^jz_{n_i},\quad Q=M^{(r_{i-1}+j)}(z),\quad Q'=F_{i-1}M^{(j)}(z_{n_i}).
\]
The futures agree through the remaining block and $n_i$ further $\mathsf a$ symbols, so
\[
\|M^{(j)}(\sigma^{r_{i-1}}z)-M^{(j)}(z_{n_i})\|\le K\cdot\,600^{-n_i},\qquad \|M^{(j)}(\sigma^{r_{i-1}}z)\|\le2K.
\] 
The cocycle identity and $\|F_{i-1}\|\le2$ therefore give
\[
\|Q-Q'\|\le4K\cdot600^{-n_{i-1}}+2K\cdot\,600^{-n_i}\le6K\cdot\,600^{-n_{i-1}}
\] which implies that $\|Q'\|,\|Q\|\le2K+1$.

The common past and \eqref{eq:word-closeness} yield $\|u(w)-u(w^\prime)\|_\infty+\|v(w)-v(w^\prime)\|_\infty\le3\cdot600^{-n_{i-1}}.$

Thus we have
\[\|u(w)-u(w^\prime)\|_\infty+\|v(w)-v(w^\prime)\|_\infty+\|Q-Q'\|\le D\cdot 600^{-n_{i-1}}<\rho_i.\]

By the definition of $\rho_i$ and~\eqref{eq:factor-map}, we therefore obtain, for every $i\ge2$ and every integer $j$ with $0\le j<L_i(n_i+1)$,
\begin{equation}\label{eq:return-continuity}
\left|\max_{T_{r_{i-1}+j}(z)\le t\le T_{r_{i-1}+j+1}(z)}F(g_tC(z,I_2)\Z^3)-\max_{T_j(z_{n_i})\le t\le T_{j+1}(z_{n_i})}F\bigl(g_tC(z_{n_i},F_{i-1})\Z^3\bigr)\right|<e_i.
\end{equation}
Taking the maximum over $0\le j<L_i(n_i+1)$ in~\eqref{eq:return-continuity} gives
\[
\left|\max_{s_{i-1}\le t\le s_i}F(g_tC(z,I_2)\Z^3)-\max_{0\le t\le L_i\tau_{n_i}}F\bigl(g_tC(z_{n_i},F_{i-1})\Z^3\bigr)\right|<e_i.
\]
By~\eqref{eq:chosen-repetition-length},
together with $|\Delta_{n_i}(F_{i-1})-A|<e_i$, this yields
\[
\left|\max_{s_{i-1}\le t\le s_i}F(g_tC(z,I_2)\Z^3)-A\right|<3e_i.
\]
Since $s_i\ge r_i\to\infty$ and $e_i\to0$, these estimates prove~\eqref{eq:prescribed-limsup}; the first block does not affect the limit.

\paragraph{Step 4: Forgetting the past.}
Put $x=u(z)$ and $\iota=1+v(z)u(z)$.
The chart formula gives
\[
C(z,I_2)
=g_{-\frac13\log\iota}
\begin{pmatrix}1&v(z)/\iota\\0&I_2\end{pmatrix}h_x,
\]
while
\[
g_t\begin{pmatrix}1&v(z)/\iota\\0&I_2\end{pmatrix}g_{-t}
=\begin{pmatrix}1&e^{-3t}v(z)/\iota\\0&I_2\end{pmatrix}
\longrightarrow I_3.
\]
Thus, after a fixed time translation, the two flowing lattices differ by left multiplication by matrices tending to the identity.
Since $F$ is proper, Equation~\eqref{eq:prescribed-limsup} shows the forward orbit is in a compact set, so uniform continuity of $F$ gives
\[
\limsup_{t\to\infty}F(g_th_x\Z^3)=A.
\]
\end{proof}

\subsection{Level sets are uncountable}
From our construction, we also see that the level sets are uncountable in a Hall ray. 
\begin{proposition}\label{prop:uncountable-levels}
For every height function $F$, the constant $T$ in
Theorem~\ref{thm:proper-ray} can be chosen so that, for every $A>T$,
the set
\[
\left\{x\in\R^2:
\limsup_{t\to\infty}F(g_th_x\Z^3)=A
\right\}
\]
is uncountable.
Consequently, every level $0<\rho<\rho_0$ in
Theorem~\ref{thm:main} is attained by uncountably many vectors.
\end{proposition}

\begin{proof}
Fix $T$ as in the proof of Theorem~\ref{thm:proper-ray}, and fix any $A>T$. We follow the construction in
Section~\ref{sec:realization}.
The two limits immediately preceding Equation~\eqref{eq:chosen-repetition-length} show that, after choosing $n_i$, there are infinitely many integers $L_i\ge i$ satisfying
\eqref{eq:chosen-repetition-length}.
Retain two distinct such choices at each stage and continue the induction separately along each branch.
Every infinite branch produces a word
\[
z[0,\infty)
=(\mathsf a^{n_1}\mathsf b)^{L_1}
 (\mathsf a^{n_2}\mathsf b)^{L_2}\cdots,
\qquad n_1<n_2<\cdots,
\]
where all negative symbols are $\mathsf a$.
Steps~3--4 of that proof apply to every branch and give
$\limsup_{t\to\infty}F(g_th_{u(z)}\Z^3)=A$.

Distinct branches give distinct futures $z[0,\infty)$.
The injectivity and disjointness of the maps $\Phi_s$ in Lemma~\ref{lem:branches} imply that distinct futures give distinct vectors $u(z)$.
We therefore obtain an injection of $\{0,1\}^{\N}$ into the level set $\{u:\limsup_{t\to\infty}F(g_th_{u}\Z^3)=A\}$.

Finally, take $F=1/\sys$ and $\rho_0=T^{-3/2}$. Lemma~\ref{lem:dani}, with $A=\rho^{-2/3}$, shows that the level set $\{\theta:\liminf_{q\to\infty}q^{1/2}\|q\theta\|_{\T^2}=\rho\}$ is uncountable.
\end{proof}


\begin{thebibliography}{99}

\bibitem{Agin}
A. Agin,
Constructing displacement vectors,
\emph{Comb. Number Theory} \textbf{13} (2024), no.~3, 239--264.
\href{https://doi.org/10.2140/cnt.2024.13.239}{doi:10.2140/cnt.2024.13.239}.

\bibitem{AginWeiss}
A. Agin and B. Weiss,
The Dirichlet spectrum,
to appear in \emph{Selecta Math. (N.S.)},
\href{https://arxiv.org/abs/2412.05858}{arXiv:2412.05858} (2024).

\bibitem{Akhunzhanov}
R. K. Akhunzhanov,
Vectors of a given Diophantine type. II,
\emph{Sb. Math.} \textbf{204} (2013), no.~4, 463--484.
\href{https://doi.org/10.1070/SM2013v204n04ABEH004308}{doi:10.1070/SM2013v204n04ABEH004308}.

\bibitem{AkhunzhanovMoshchevitin}
R. K. Akhunzhanov and N. G. Moshchevitin,
Vectors of given Diophantine type,
\emph{Math. Notes} \textbf{80} (2006), no.~3--4, 318--328.
\href{https://doi.org/10.1007/s11006-006-0143-2}{doi:10.1007/s11006-006-0143-2}.

\bibitem{AkhunzhanovShatskov}
R. K. Akhunzhanov and D. O. Shatskov,
On Dirichlet spectrum for two-dimensional simultaneous
Diophantine approximation,
\emph{Mosc. J. Comb. Number Theory} \textbf{3} (2013),
no.~3--4, 5--23.
\href{https://arxiv.org/abs/1306.1876}{arXiv:1306.1876}.

\bibitem{ArtigianiMarcheseUlcigrai}
M. Artigiani, L. Marchese, and C. Ulcigrai,
The Lagrange spectrum of a Veech surface has a Hall ray,
\emph{Groups Geom. Dyn.} \textbf{10} (2016), no.~4, 1287--1337.
\href{https://doi.org/10.4171/GGD/384}{doi:10.4171/GGD/384}.

\bibitem{PersistentHall}
M. Artigiani, L. Marchese, and C. Ulcigrai,
Persistent Hall rays for Lagrange spectra at cusps of Riemann surfaces,
\emph{Ergodic Theory Dynam. Systems} \textbf{40} (2020),
no.~8, 2017--2072.
\href{https://arxiv.org/abs/1710.02042}{arXiv:1710.02042}.



\bibitem{Cassels}
J. W. S. Cassels,
\emph{An Introduction to the Geometry of Numbers},
Classics in Mathematics,
Springer-Verlag, Berlin, 1997.
\href{https://doi.org/10.1007/978-3-642-62035-5}{doi:10.1007/978-3-642-62035-5}.

\bibitem{CheungChevallier}
Y. Cheung and N. Chevallier,
L\'evy--Khintchin theorem for best simultaneous Diophantine approximations,
\emph{Ann. Sci. \'Ec. Norm. Sup\'er.} (4) \textbf{57} (2024),
no.~1, 185--240.
\href{https://arxiv.org/abs/1906.11173}{arXiv:1906.11173}.

\bibitem{CrispMoranPollington}
D. J. Crisp, W. Moran, and A. D. Pollington,
The inhomogeneous Hall's ray,
preprint (2012),
\href{https://arxiv.org/abs/1203.4295}{arXiv:1203.4295}.

\bibitem{Dani}
S. G. Dani,
Divergent trajectories of flows on homogeneous spaces
and Diophantine approximation,
\emph{J. Reine Angew. Math.} \textbf{359} (1985), 55--89.

\bibitem{ErazoEtAl}
H. Erazo, R. Guti\'errez-Romo, C. G. Moreira, and S. Roma\~na,
Fractal dimensions of the Markov and Lagrange spectra near $3$,
\emph{J. Eur. Math. Soc.} \textbf{28} (2026), no.~9, 3983--4046.
\href{https://doi.org/10.4171/JEMS/1545}{doi:10.4171/JEMS/1545}.

\bibitem{Freiman}
G. A. Freiman,
\emph{Diofantovy priblizheniya i geometriya chisel (zadacha Markova)},
Kalininskii Gosudarstvennyi Universitet, Kalinin, 1975 (in Russian).

\bibitem{Hall}
M. Hall, Jr.,
On the sum and product of continued fractions,
\emph{Ann. of Math.} (2) \textbf{48} (1947), no.~4, 966--993.
\href{https://doi.org/10.2307/1969389}{doi:10.2307/1969389}.

\bibitem{HussainSchleischitzWard}
M. Hussain, J. Schleischitz, and B. Ward,
On the Folklore set and Dirichlet spectrum for matrices,
preprint (2024; revised 2025),
\href{https://arxiv.org/abs/2402.13451v3}{arXiv:2402.13451v3}.

\bibitem{KatokHasselblatt}
A. Katok and B. Hasselblatt,
\emph{Introduction to the Modern Theory of Dynamical Systems},
Encyclopedia of Mathematics and its Applications, vol.~54,
Cambridge University Press, Cambridge, 1995.
\href{https://doi.org/10.1017/CBO9780511809187}{doi:10.1017/CBO9780511809187}.

\bibitem{Kleinbock}
D. Kleinbock,
Density of the multidimensional Lagrange spectrum,
preprint (2026),
\href{https://arxiv.org/abs/2608.30735v2}{arXiv:2608.30735v2}.

\bibitem{KleinbockRao}
D. Kleinbock and A. Rao,
Abundance of Dirichlet-improvable pairs with respect to arbitrary norms,
\emph{Mosc. J. Comb. Number Theory} \textbf{11} (2022),
no.~1, 97--114.

\bibitem{Markov}
A. Markoff,
Sur les formes quadratiques binaires ind\'efinies,
\emph{Math. Ann.} \textbf{15} (1879), 381--406.
\href{https://doi.org/10.1007/BF02086269}{doi:10.1007/BF02086269}.

\bibitem{Moreira}
C. G. Moreira,
Geometric properties of the Markov and Lagrange spectra,
\emph{Ann. of Math.} (2) \textbf{188} (2018), no.~1, 145--170.
\href{https://doi.org/10.4007/annals.2018.188.1.3}{doi:10.4007/annals.2018.188.1.3}.

\bibitem{ParkkonenPaulin}
J. Parkkonen and F. Paulin,
Prescribing the behaviour of geodesics in negative curvature,
\emph{Geom. Topol.} \textbf{14} (2010), no.~1, 277--392.
\href{https://doi.org/10.2140/gt.2010.14.277}{doi:10.2140/gt.2010.14.277}.

\bibitem{Schleischitz}
J. Schleischitz,
Exact uniform approximation and Dirichlet spectrum
in dimension at least two,
\emph{Selecta Math. (N.S.)} \textbf{29} (2023),
no.~5, Paper No.~86.
\href{https://doi.org/10.1007/s00029-023-00889-0}{doi:10.1007/s00029-023-00889-0}.

\bibitem{ShapiraWeiss}
U. Shapira and B. Weiss,
Geometric and arithmetic aspects of approximation vectors,
to appear in \emph{J. Anal. Math.},
\href{https://arxiv.org/abs/2206.05329}{arXiv:2206.05329} (2022).

\bibitem{Ward}
B. Ward,
Distribution of points near the origin in the $d$-dimensional
Lagrange spectrum,
preprint (2026),
\href{https://arxiv.org/abs/2609.05311v2}{arXiv:2609.05311v2}.

\end{thebibliography}
\end{document}